\documentclass[11pt]{amsart}
\usepackage{amsmath} 
\usepackage{amsthm}
\usepackage{amssymb} 
\usepackage{amsfonts}
\usepackage{graphicx}
\usepackage{color}
\usepackage[normalem]{ulem}
\usepackage{comment}
\usepackage{enumitem}

\usepackage{caption,subcaption,pinlabel}
\usepackage{graphics}

\newcommand{\pt}{\text{pt}}

\newtheorem{theorem}{Theorem}[section]
\newtheorem{lemma}[theorem]{Lemma}
\newtheorem{proposition}[theorem]{Proposition}
\newtheorem{question}[theorem]{Question}

\newtheorem{corollary}[theorem]{Corollary}

\theoremstyle{definition}
\newtheorem{definition}[theorem]{Definition}

\newtheorem{construction}[theorem]{Construction}

\newtheorem{goal}[theorem]{Goal}
\newtheorem{remark}[theorem]{Remark}
\newtheorem{operation}[theorem]{Operation}

\newcommand{\Z}{\mathbb{Z}}
\newcommand{\N}{\mathbb{N}}
\newcommand{\Q}{\mathbb{Q}}
\newcommand{\R}{\mathbb{R}}

\newcommand{\F}{\mathcal{F}}
\newcommand{\G}{\mathcal{G}}
\let\int\relax
\newcommand{\int}{\mathring}

\newcommand{\boundary}{\partial}

\DeclareMathOperator{\id}{{id}}
\DeclareMathOperator{\lk}{{lk}}

    \author[Maggie Miller]{Maggie Miller}
    \title[The effect of link Dehn surgery on the Thurston norm]{The effect of link Dehn surgery on the Thurston norm}
    
    \address{Princeton University, Princeton, NJ 08540, USA}\email{maggiem@math.princeton.edu}
\begin{document}

\begin{abstract}
Let $L$ be an $n$-component link ($n>1$) with pairwise nonzero linking numbers in a rational homology $3$-sphere $Y$. Assume the link complement $X:=Y\setminus\nu(L)$ has nondegenerate Thurston norm. In this paper, we study when a Thurston norm-minimizing surface $S$ properly embedded in $X$ remains norm-minimizing after Dehn filling all boundary components of $X$ according to $\boundary S$ and capping off $\boundary S$ by disks. In particular, for $n=2$ the capped-off surface is norm-minimizing when $[S]$ lies outside of a finite set of rays in $H_2(X,\boundary X;\R)$. When $Y$ is an integer homology sphere this gives an upper bound on the number of surgeries on $L$ which may yield $S^1\times S^2$. The main techniques come from Gabai's proof of the Property R conjecture and related work. 
\end{abstract}

\maketitle
\section{Introduction}\label{sec:intro}

The {\emph{Thurston norm}}, introduced by Thurston~\cite{thurston}, is a pseudonorm on the second homology of a compact $3$-manifold $M$. Specifically, if $M$ is a compact $3$-manifold (with or without boundary), then ``Thurston norm'' refers to one of two canonical continuous functions \[x_M:H_2(M;\R)\to\R,\hspace{.25in}\text{or}\hspace{.25in} x_M:H_2(M,\boundary M;\R)\to\R.\] The Thurston norm $x_M$ on $M$ is degenerate if $x_M(\alpha)=0$ for some $\alpha\neq 0$, and nondegenerate otherwise. Unless otherwise specified, we will always take homology to have real coefficients; we are uninterested in torsion classes. In this paper, we study the effect of Dehn surgery on an oriented link $L$ on the Thurston norm on $H_2(Y\setminus\nu(L),\boundary(Y\setminus\nu(L));\R)$, where $Y$ is a rational homology $3$-sphere. In particular, we prove the following theorem.

\begin{theorem}\label{mingenuscor}
Let $L=L_1\sqcup L_2$ be a $2$-component link in a rational homology $3$-sphere $Y$. Assume $\lk(L_1,L_2)\neq 0$ and that $X:=Y\setminus\nu(L)$ has nondegenerate Thurston norm $x$. Let $S$ be a norm-minimizing surface in $X$. Take $[S]$ to be primitive and in the cone $C$ on the interior of some face of the unit-norm ball of $x$.

Let $\widehat{X}$ be the closed $3$-manifold obtained by Dehn filling each component of $\boundary X$ according to the slope of $S\cap\boundary X$. Each component of $\boundary S$ can be capped off by a disk (in a Dehn-filling solid torus) to create a closed surface $\widehat{S}$ in $\widehat{X}$.

If $\widehat{S}$ is not norm-minimizing, then either $g(S)=1$ or the genus $g([S])$ is minimal among all homology classes in $C$.  In particular, if $[S']$ is a primitive class in $C$ and $\widehat{S'}$ is also not norm-minimizing, then $g([S])=g([S'])$.
\end{theorem}

We use the convention that in a rational homology sphere $Y$, a Seifert surface for a knot $K$ is a surface $S$ so that $[\boundary S]$ generates the kernel of $i_*:H_1(\boundary (Y\setminus\nu(K);\Z)\to H_1(Y\setminus\nu(K);\Z)$. The slope of $\boundary S$ (and any multiple) is zero. Let $c$ be the smallest positive integer with $c[K]=0$ in $H_1(Y;\Z)$. If $K'$ is another knot in $Y\setminus\nu(K)$, we say the linking number of $K$ and $K'$ is $\lk(K,K')=\frac{1}{c}\langle S,K'\rangle$.

For Theorem~\ref{mingenuscor} and the remainder of Section~\ref{sec:intro},  a surface is norm-minimizing if it maximizes Euler characteristic among all homologous surfaces and includes no nullhomologous subset of components (in homology with real coefficients). We also require that on each torus component of $\boundary X$, the boundary components of a norm-minimizing surface have parallel orientations. 

We develop the language to parse and state Theorem~\ref{mingenuscor} more precisely and concisely in Section~\ref{sec:thurston}. Theorem~\ref{mingenuscor} can be stated for links with more components as well, but to compare homology classes we require them to have different boundary slopes on each boundary component of $X$.

\begin{theorem}\label{mingenuscor2}
Let $L=L_1\sqcup \cdots\sqcup L_n$ be an $n$-component link in a rational homology $3$-sphere $Y$. Assume $\lk(L_i,L_j)\neq 0$ for each $i\neq j$. Let $X:=Y\setminus\nu(L)$. Assume $X$ has nondegenerate Thurston norm.

Let $S$ be a norm-minimizing surface meeting every component of $\boundary X$ so that $[S]$ is primitive and is contained in a cone $C$ on the interior of a face of the unit-ball of $x$. 

Let $\widehat{X}$ be the closed manifold obtained by Dehn-filling each boundary component of $X$ according to the slope of $\boundary S$. Let $\widehat{S}$ be the closed surface in $\widehat{X}$ obtained from $S$ by capping off each boundary component of $S$ by a disk in a Dehn-filling solid torus. Then at least one of the following is true:
\begin{itemize}
\item $g(S)=1$,
\item $\widehat{S}$ is norm-minimizing,
\item $g(S)\le g(\beta)$ whenever $\beta\in C$.
\end{itemize}
\end{theorem}

%

We will prove Theorem~\ref{mingenuscor2} in Section~\ref{sec:ncomp}.

\begin{remark}
In Theorem~\ref{mingenuscor2} (and later observations about link complements), we require a surface to meet every boundary component of $X$ because we are interested in surgering every component of a link. If a surface $S$ in $X=Y\setminus\nu(L)$ does not meet every boundary component of $X$, then we may restrict $L$ to a sublink $L'$ so that $S\subset X':=Y\setminus\nu(L')$ and $S$ does meet every boundary component of $X'$. Then we may apply Theorem~\ref{mingenuscor2} (or other relevant theorems) to the link $L'$ in $Y$.

Let $\mathcal{S}_P$ be the set of norm-minimizing surfaces not meeting some boundary component $P$ of $X$. Work of Sela~\cite{sela} implies that for all but finitely many choices of slope on $P$, filling $P$ and capping off a surface in $\mathcal{S}_P$ will yield a norm-minimizing surface. We do not separately describe Sela's theorem as it requires some more technical discussion, but the reader may refer to~\cite{sela} (expanding on work of Gabai~\cite{ft3m2}) if interested in Dehn fillings on boundary torus components which do {\emph{not}} meet some norm-minimizing surface.
\end{remark}

The primary motivation for Theorem~\ref{mingenuscor} is to study the total number of norm-minimizing surfaces $S\subset X$ (up to homology) so that $\widehat{S}$ are not norm-minimizing. We state the corresponding result first for $2$-component links and then in generality; this is not strictly necessary but aids in readability.

\begin{theorem}\label{2compthm}
Let $L=L_1\sqcup L_2$ be a $2$-component link in a rational homology $3$-sphere $Y$. Assume $\lk(L_1,L_2)\neq 0$ and that $X:=Y\setminus\nu(L)$ has nondegenerate Thurston norm.

Let $S$ be a norm-minimizing surface in $X$ with $[S]$ primitive. Let $\widehat{X}$ be the closed $3$-manifold obtained from $X$ by Dehn-filling both components $P_i=\boundary(\nu(L_i))$ of $\boundary X$ according to the slope $\boundary S\cap P_i$. Let $\widehat{S}$ be the closed surface in $\widehat{X}$ obtained by capping off each component of $\boundary S$ with a disk in one of the Dehn-filling solid tori.

There exists a finite set $E\subset H_2(X,\boundary X;\mathbb{Z})$ so that if $[S]\not\in E$, then $\widehat{S}$ is norm-minimizing.

\end{theorem}

In the conclusion of Theorem~\ref{2compthm}, we also obtain explicit upper bounds on $|E|$ which depend on $Y$ and $L$. Some bounds will be described in Corollaries~\ref{scholium},~\ref{cor1} and~\ref{cor2}. Theorem~\ref{ncompthm} is related to the Property R theorem of Gabai~\cite{ft3m3}: surgery on a nontrivial knot in $S^3$ cannot yield $S^1\times S^2$. (We discuss this further in Section~\ref{sec:thurston}.)

\begin{corollary}\label{proprcor}
Let $L=L_1\sqcup L_2$ be a $2$-component link in a homology $3$-sphere $Y$. Assume $\lk(L_1\sqcup L_2)\neq 0$ and that any annulus properly embedded in $X:=Y\setminus\nu(L)$ represents $0\in H_2(X,\boundary X;\R)$. 
Then there are at most finitely many surgeries on $L$ which yield $S^1\times S^2$. We give explicit bounds on this finite number in Corollaries~\ref{scholium},~\ref{cor1} and~\ref{cor2}. \end{corollary}

In Corollary~\ref{proprcor}, we are implicitly using the fact that there are finitely many nontrivial elements of $\alpha\in H_2(X;\R)$ whose norm-minimizing representatives have genus zero. We prove this later in Proposition~\ref{g0prop}.

We prove Theorem~\ref{2compthm} from Theorem~\ref{mingenuscor} very quickly in Section~\ref{sec:sutured}. We give examples of $2$-component links in $S^3$ in which $E$ is nonempty in the setting of Theorem~\ref{2compthm} in Figures~\ref{fig:example1} and~\ref{fig:example2}. (In both cases, we obtained the Thurston norm on $S^3\setminus\nu(L)$ from work of McMullen~\cite{mcmullen}, who computed the Thurston norm on complements of nearly all links with at most nine crossings, including these two examples.)

\begin{figure}
\begin{centering}
\labellist
\small\hair 2pt
\pinlabel $L_1$ at 0 175
\pinlabel $L_2$ at 130 175
\pinlabel $S_1$ at 200 100
\pinlabel $S_2$ at 445 100
\pinlabel $\textcolor{red}{[S_1]}$ at 685 105
\pinlabel $\textcolor{blue}{[S_2]}$ at 565 175
\endlabellist
\includegraphics[width=100mm]{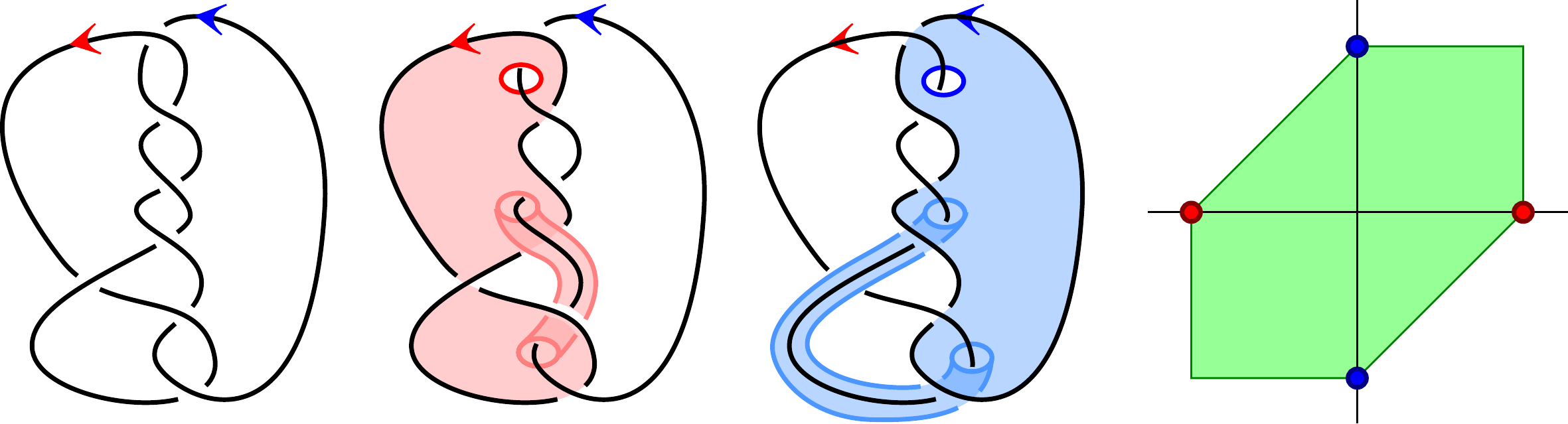}
\caption{{\bf{First:}} A $2$-component link $L=L_1\sqcup L_2$ in $S^3$. {\bf{Second:}} A norm-minimizing surface $S_1$ in the homology class of a punctured Seifert surface for $L_1$. {\bf{Third:}} A norm-minimizing surface in the homology class of a punctured surface for $L_2$. {\bf{Fourth:}} The unit ball of the Thurston norm on $S^3\setminus\nu(L)$, in which we indicate $[\pm S_1]$ and $[\pm S_2]$. The surfaces $\widehat{S_1}$ and $\widehat{S_2}$ are not norm-minimizing.}\label{fig:example1}
\end{centering}
\end{figure}
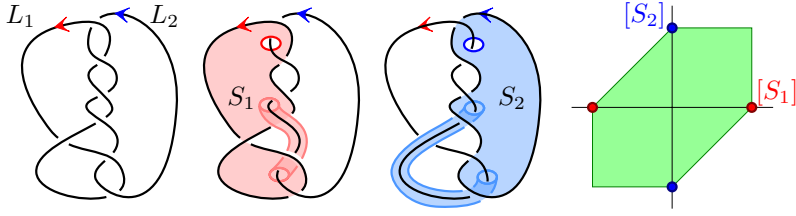

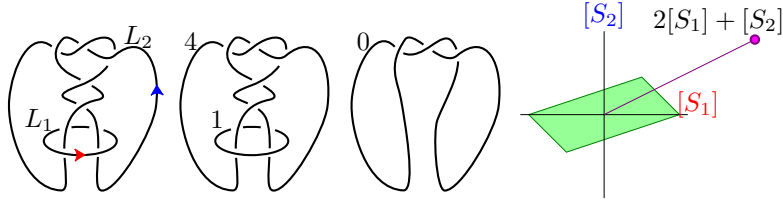
\begin{figure}
\begin{centering}
\labellist
\small\hair 2pt
\pinlabel $L_1$ at 30 75
\pinlabel $L_2$ at 125 155
\pinlabel $1$ at 200 75
\pinlabel $4$ at 175 150
\pinlabel $0$ at 340 150
\pinlabel $2[S_1]+[S_2]$ at 680 170
\pinlabel $\textcolor{red}{[S_1]}$ at 660 90
\pinlabel $\textcolor{blue}{[S_2]}$ at 570 175
\endlabellist
\includegraphics[width=100mm]{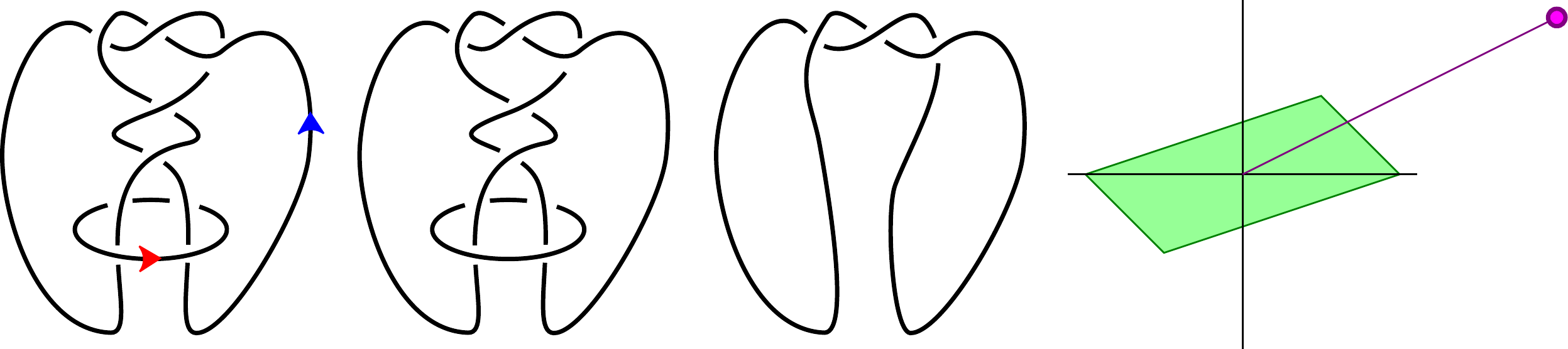}
\caption{{\bf{First:}} A $2$-component link $L=L_1\sqcup L_2$ in $S^3$. {\bf{Second:}} Let $S_i$ be a surface in the homology class of a punctured Seifert surface for $L_i$. We indicate the boundary slopes of $2[S_1]+[S_2]$. {\bf{Third:}} Surgery on $L_i$ according to these slopes yields $S^1\times S^2$. {\bf{Fourth:}}  The unit ball of the Thurston norm on $S^3\setminus\nu(L)$, in which we highlight $2[S_1]+[S_2]$. A norm-minimizing surface $S$ in this homology class is genus-$1$ with three boundary components, but $[\widehat{S}]$ is represented by a $2$-sphere, so the torus $\widehat{S}$ is not norm-minimizing.}\label{fig:example2}
\end{centering}
\end{figure}

\begin{theorem}\label{ncompthm}
Let $L=L_1\sqcup\cdots\sqcup L_n$ be an $n$-component link $(n>1)$ in a rational homology $3$-sphere $Y$. Assume $\lk(L_i,L_j)\neq 0$ for $i\neq j$ and that any annulus properly embedded in $X:=Y\setminus\nu(L)$ represents $0\in H_2(X,\boundary X;\R)$.

Let $S$ be a norm-minimizing surface in $X$ meeting every component of $\boundary S$. Let $\widehat{X}$ be the closed manifold obtained from $X$ by Dehn filling $X$ according to $\boundary S$, and let $\widehat{S}\subset\widehat{X}$ be the closed surface obtained from capping off each boundary component of $S$ by a disk within the Dehn-filling solid tori.

Let $\widetilde{Y}$ be the $3$-manifold obtained from $Y$ by surgering $Y$ along $L_3\sqcup\cdots\sqcup L_n$ according to $\boundary S$.

There exists an $(n-2)$-dimensional set of rays $E$ from the origin of $H_2(X,\boundary X;\R)\cong\R^n$ so that if $[S]\not\in E$, then either $\widehat{S}$ is norm-minimizing or $\widetilde{Y}\setminus\nu(L_1\sqcup L_2)$ has degenerate Thurston norm.
\end{theorem}

\begin{remark}
In Theorem~\ref{ncompthm}, we may reorder the components of $L$ to apply the theorem, if this causes $\widetilde{Y}\setminus\nu(L_1\sqcup L_2)$ to have nondegenerate Thurston norm. 
\end{remark}

In the statement of Theorem~\ref{ncompthm}, we describe $E$ as a set of rays in $H_2(X,\boundary X;\R)$ because we are primarily interested in primitive homology classes. (We remind the reader that a primitive class $\alpha$ is one which is integral and not equal to $c\beta$ for any integral class $\beta$ and integer $c>1$.) If $\widehat{cS}$ is not norm-minimizing, then neither is $\widehat{S}$. Thus, $E$ lies in an $(n-2)$-dimensional subcomplex of $(H_2(X,\boundary X;\R)\setminus0)/($positive scalar multiplication) $\cong S^{n-1}$. For example, if $n=3$, then $E$ lies in a graph embedded in $S^2$. We could have alternately phrased the conlusion of Theorem~\ref{ncompthm} as, ``for any $q_1,\ldots, q_{n-2}\subset\Q\cup\{\pm\infty\}$, there exists a finite set $E\subset(\Q\cup\{\pm\infty\})^2$ so that if $\boundary S$ has slope $q_i$ on $\boundary\overline{\nu(L_i)}$ and $(q_{n-1},q_n)\not\in E$, then either $\widetilde{Y}\setminus\nu(L_1\sqcup L_2)$ has degenerate Thurston norm or $\widehat{S}$ is norm-minimizing.'' We believe the given conclusion to be more illuminating, but we are happy for the reader to instead think about surgery slopes.

The proof of Theorems~\ref{mingenuscor},~\ref{mingenuscor2},~\ref{2compthm}, and~\ref{ncompthm} are motivated by Gabai's constructions of taut foliations and sutured manifolds. In particular, we use many important theorems of Gabai which were essential in the proof of Property R. We give specific references and background in Section~\ref{sec:sutured}.

Note that increasing the number of link components allows us to make the statement of Corollary~\ref{proprcor} about homology $3$-spheres, even though the Property R theorem about knots is not generally true in homology $3$-spheres. For instant, obtain $Y$ by surgering $S^1\times S^2$ along a curve homologous (but not isotopic) to $S^1\times \pt$, and let the knot $L$ be a core of the surgery solid torus. See Figure~\ref{fig:mazur} for an explicit example.

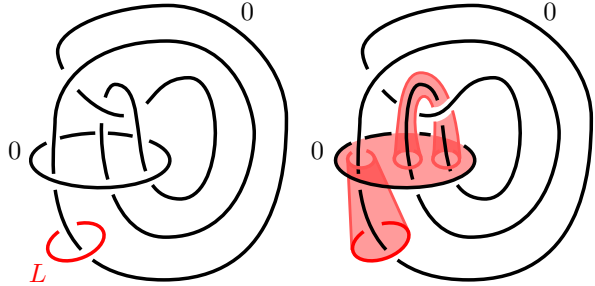
\begin{figure}
\labellist
\small\hair 2pt
\pinlabel $0$ at -8 75
\pinlabel $0$ at 125 155
\pinlabel $0$ at 165 75
\pinlabel $0$ at 298 155
\pinlabel $\textcolor{red}{L}$ at 5 5
\endlabellist
\begin{centering}
\includegraphics[width=75mm]{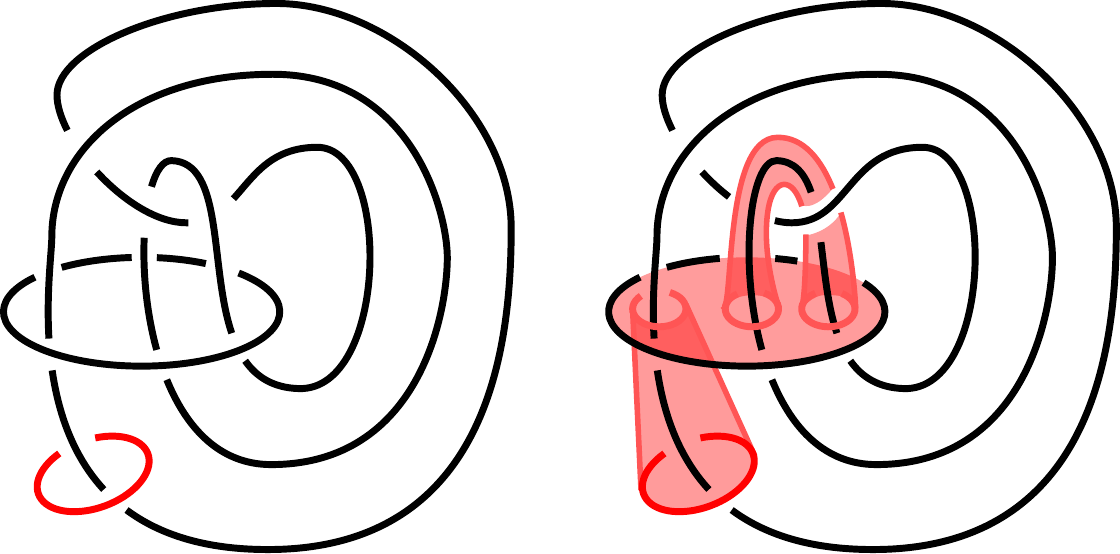}
\caption{{\bf{Left:}} A surgery diagram for the Brieskorn sphere $Y:=\Sigma(2,5,7)$. Zero-framed surgery on the knot $L\subset Y$ yields $S^1\times S^2$, but $L$ does not bound a disk in $Y$. {\bf{Right:}} A norm-minimizing surface in $Y\setminus\nu(L)$.}\label{fig:mazur}
\end{centering}
\end{figure}

Finally, we wish to remark that assumptions on linking number are generally just to avoid trivialities. Very roughly, the main arguments in the proofs of Theorems~\ref{mingenuscor},~\ref{mingenuscor2},~\ref{2compthm}, and~\ref{ncompthm} will come from understanding the intersections of two nonhomologous minimal surfaces. We take the ambient manifold $Y$ to be a rational homology $3$-sphere and $\lk(L_i,L_j)\neq 0$ so that the boundary data of a surface $S$ in $Y\setminus\nu(L)$ determines the homology class of $S$. That is, so that $\boundary:H_2(Y\setminus\nu(L),\boundary(Y\setminus\nu(L));\Z)\to H_1(\boundary(Y\setminus\nu(L));\Z)$ is injective. However, we could have stated the conclusion of Theorem $\ref{ncompthm}$ to be ``there exists an $E\subset$(an $(n-2)$-dimensional complex in $(\Q\cup\{\infty\})^n)$ so that if $\boundary S$ meets every $P_i$ and $\boundary S=(\boundary S\cap P_1,\ldots,\boundary S\cap P_n)$ is not in $E$, then $\widehat{S}$ is norm-minimizing or the Thurston norm becomes degenerate after surgery on $(n-2)$ components of $L$.'' This holds perfectly well when some (or all) of the pairwise linking numbers $\lk(L_i,L_j)$ are zero (If $L$ is split, then divide $L$ into two split sublinks and induct), as then the image of $\boundary$ is at most $(n-1)$-dimensional, and contains an at most $(n-2)$-dimensional class of primitive elements.  As a simple example, consider a $2$-component link $L_1\sqcup L_2$ in $S^3$ with linking number zero, e.g. the Whitehead link.  Any norm-minimizing surface $S$ in $S^3\setminus\nu(L)$ meeting $\boundary(\overline{\nu(L_i)})$ does so in curves of slope $0$, so we may let $E\subset (\R\cup\infty)^2$ be the finite set $\{(0,0)\}$ and be done.

We leave the reader with the following natural question.

\begin{question}
Let $L$ be as in Theorem~\ref{2compthm} or Theorem~\ref{ncompthm}. Take the rational homology $3$-sphere $Y$ to actually be the $3$-sphere. Using $Y=S^3$, can we further restrict the set of norm-minimizing surfaces $S$ with the property that $\widehat{S}$ is not norm-minimizing?
\end{question}

\subsection*{Organization}
We break the paper into the following sections.

\begin{itemize}[label={},leftmargin=*]
\item {\bf{Section~\ref{sec:thurston}:}} We introduce the Thurston norm and some of its basic properties.
\item {\bf{Section~\ref{sec:sutured}:}} We discuss sutured manifolds and foliations. We restate the main theorems and prove Theorem~\ref{2compthm} from Theorem~\ref{mingenuscor}.
\item {\bf{Section~\ref{sec:graphs}:}} We introduce fat-vertex graphs, which will be used to describe certain intersections of surfaces. 
\item {\bf{Section~\ref{sec:proof}:}} We prove Theorem~\ref{mingenuscor}, using tools from Sections~\ref{sec:thurston}--\ref{sec:graphs}.
\item {\bf{Section~\ref{sec:ncomp}:}} We prove Theorem~\ref{mingenuscor2}. We then prove Theorem~\ref{ncompthm} by inducting on Theorem~\ref{2compthm}.
\end{itemize}

\subsection*{Acknowledgements}

The author would like to thank her graduate advisor, David Gabai, for suggesting this question in the spring of 2017 and for many helpful conversations over a very long period of time. The author is a fellow in the National Science Foundation Graduate Research Fellowship program, under Grant No. DGE-1656466.

\section{The Thurston norm}\label{sec:thurston}

In this section, we review the definition of the Thurston norm and some of its basic properties.

The Thurston norm can be defined by its valuation on integral homology classes, and then extended naturally to rational and then all real homology classes. On integral homology classes, the Thurston norm is geometrically motivated.

\begin{definition}[Thurston~\cite{thurston}]\label{def:thurston}
Let $M$ be a compact $3$-manifold. Given an oriented surface $S$ smoothly embedded in $M$, we define \[\chi^+(S)=\begin{cases}\max\{-\chi(S),0\}&S\text{ is connected,}\\\sum_{i=1}^n\chi^+(S_i)&S=\sqcup_{i=1}^n S_i.\end{cases}\]

Now let $\alpha$ be an integral element of $H_2(M;\R)$ or $H_2(M,\boundary M;\R)$, where $M$ is a compact $3$-manifold. Define:
\[x_M(\alpha)=\min\{\chi^+(S)\mid S\text{ is a surface embedded in $M$ with $[S]=\alpha$}\}.\]

In words, if $S$ is the surface representing $\alpha$ with least negative Euler characteristic, then $x_m(\alpha)=|\chi(S)|$. However, we do not allow nullhomologous disk or $2$-sphere components to artificially increase Euler characteristic, hence the definition of $\chi^+$. Also, if $S$ happens to have positive Euler characteristic, then we say $x_M(\alpha)=0$ rather than a negative number, as we wish for $x_M$ to extend to a pseudonorm.

Now let $\beta$ be a rational element of $H_2(M,\R)$ or $H_2(M,\boundary M;\R)$. Say $q\beta$ is an integral homology class, where $q\in\N$. Then \[x_M(\beta)=\frac{1}{q}x_M(q\beta).\]

Finally, let $\gamma$ be any element of $H_2(M,\R)$ or $H_2(M,\boundary M;\R).$ Suppose $\gamma$ is approximated by rational homology classes $\beta_1,\beta_2,\ldots$, so $\gamma=\lim_{n\to\infty}\beta_n$. Then \[x_M(\gamma)=\lim_{n\to\infty}x_M(\beta_n).\]
\end{definition}
In Definition~\ref{def:thurston}, we implicitly use the fact that any integral second homology class $\alpha$ in a $3$-manifold $M$ can be represented by an oriented surface $S$ properly embedded in $M$.

\begin{definition}
Let $M$ be a compact $3$-manifold. Let $S$ be an oriented surface representing $\alpha\in H_2(M,\R)$ or $H_2(M,\boundary M;\R)$, with $\alpha\neq 0$. We say that $S$ is {\emph{norm-minimizing}} if the following conditions hold:

\begin{itemize}
\item $S$ has no nullhomologous subset of components,
\item $\chi^+(S)=x_M(\alpha)$,
\item If $\chi(S)=0$, then any $S'$ with $[S']=[S]$ and $\chi(S')>0$ has a subset of components that is nullhomologous.
\end{itemize}

In words, a surface $S$ is norm-minimizing if it maximizes Euler characteristic among homologous surfaces, discarding any nullhomologous subsets of components (to prevent adding e.g. nullhomologous $2$-sphere components to artificially increase Euler characteristic).
\end{definition}

\begin{definition}
Let $M$ be a compact $3$-manifold whose boundary is a collection of tori. Let $S$ be an oriented surface representing an integral $\alpha\in H_2(M,\boundary M;\R)$. We say that $S$ is {\emph{properly norm-minimizing}} if $S$ is norm-minimizing and on each boundary component $P$ of $M$, all components of $\boundary S\cap P$ have the same orientation (parallel on $P$).
\end{definition}

When $M$ has boundary a collection of tori, any norm-minimizing surface $S$ can be transformed into a properly norm-minimizing surface by adding tubes to $S$ which are parallel to torus boundary components of $M$. See Figure~\ref{fig:proper}. Each tube is attached to one pair of opposite-orientation boundary components of $M$, and then a neighborhood of the tube is pushed off $\boundary M$ into the interior of $M$. The resulting surface $S'$ is norm-minimizing (as $\chi(S')=\chi(S)$) but larger genus than $S$, as $S'$ is boundary-compressible at each tube attached to $S$.

\begin{figure}
\labellist
\small\hair 2pt
\pinlabel $P$ at 15 265
\pinlabel $\textcolor{red}{S}$ at 8 50
\pinlabel $\textcolor{red}{S'}$ at 393 50
\endlabellist
\begin{centering}
\includegraphics[width=75mm]{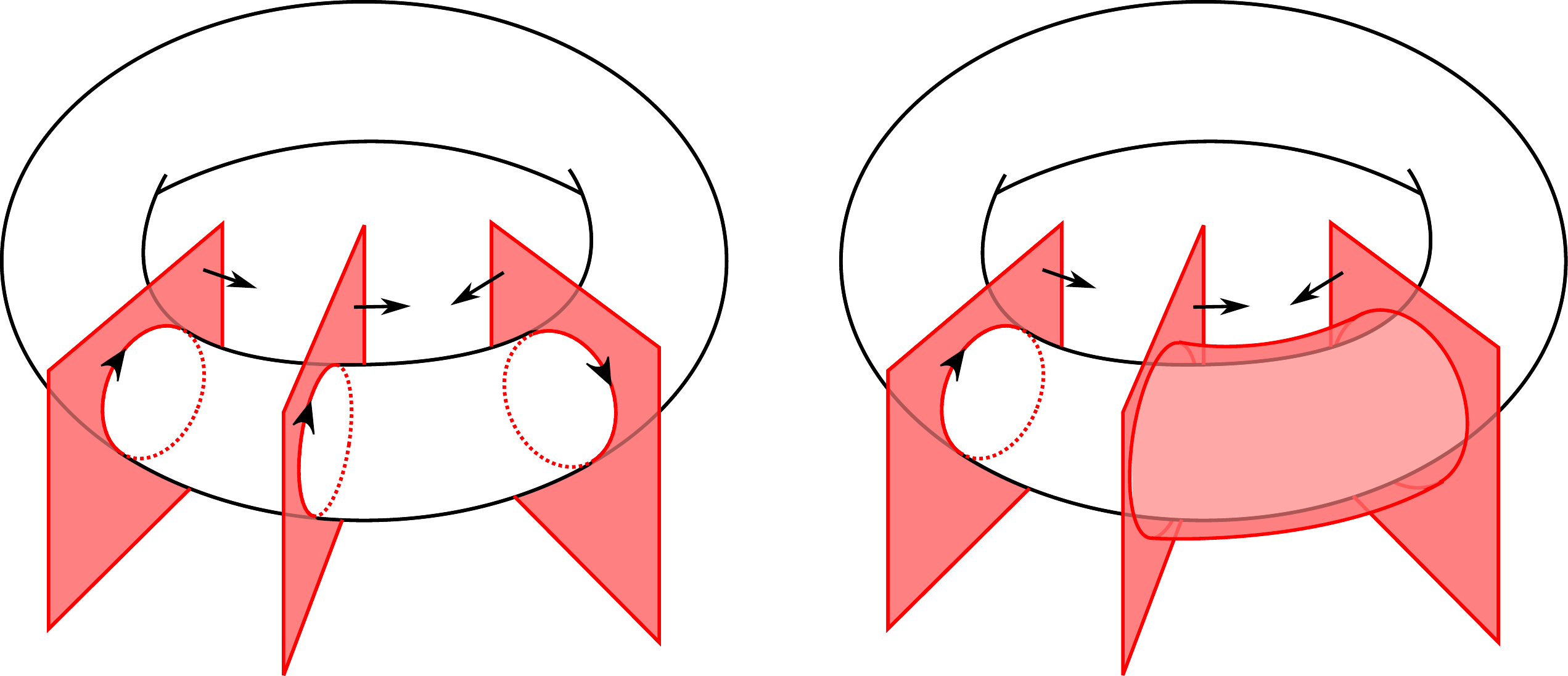}
\caption{{\bf{Left:}} A norm-minimizing surface $S$ near a torus boundary-component $P$ of $3$-manifold $X$. The components of $\boundary S$ on $P$ need not have parallel orientations. {\bf{Right:}} We increase the genus of $S$ (while preserving Euler characteristic) to find a properly norm-minimizing surface $S'$ homologous to $S$.}\label{fig:proper}
\end{centering}
\end{figure}

\begin{remark}

If $x_M(\alpha)>0$ for all $\alpha\neq 0$, then $x_M$ is a norm rather than a pseudonorm.
When $x_M$ is a norm, we say that $x_M$ is nondegenerate.
\end{remark}

In this paper, we are more interested in $H_2(M,\boundary M;\R)$ than we are in $H_2(M;\R)$. We will implicitly identify $H_2(M,\boundary M;\R)$ with $\R^n$ for appropriate $n$, where the integral lattice points correspond to integral homology classes. Thurston observed that the Thurston norm is convex and linear on rays through the origin, and that the Thurston norm is symmetric under reflection through the origin. Moreover, Thurston understood the geometry of the norm unit-ball $B_{x_M}$.
\begin{theorem}[{\cite[Theorem 2]{thurston}}]
Let $M$ be a compact $3$-manifold. Assume $x_M$ is nondegenerate. Then the unit ball $B_{x_M}\subset H_2(M,\boundary M;\R)=\R^n$ is a nondegenerate polyhedron defined by linear inequalities with integer coefficients.
\end{theorem}

\begin{definition}
Assume $x_M$ is nondegenerate. If $[S]$ is a primitive homology class and is {\emph{not}} contained in the cone on the interior of a face of $B_{x_M}$, then we will call $[S]$ a {\emph{corner}} of the Thurston norm. If $[S_1],\ldots,[S_n]$ are corners of the Thurston norm and up to positive scalar multiplication all lie in one closed face of $\boundary B_{x_M}$, then we say that $[S_1],\ldots,[S_k]$ are an {\emph{adjacent set}} of corners of the Thurston norm. When $k=2$, we may just say $[S_1]$ and $[S_2]$ are adjacent.

Note that when $x_M$ is defined on a space of dimension more than $2$, a homology class being a corner is {\emph{not}} the same as a homology class being vertex. For example, when $b_2(M)=3$ then the unit ball of $x_M:H_2(M;\R)\to\R$ is a $3$-dimensional polyhedron. We would say that a primitive homology class which is a multiple of a vertex or edge on that polyhedron is a corner of $x_M$.
\end{definition}

 Finding norm-minimizing surfaces for corners of the Thurston norm allows one to find a norm-minimizing surface for any integral homology class through cut-and-paste surgery.

\begin{definition}
Let $R$ and $T$ be oriented surfaces in a compact $3$-manifold $M$ which intersect transversely. Let $R+T$ denote the surface resulting from {\emph{cut-and-paste}} surgery on $R$ and $T$. That is, where $R$ and $T$ intersect in a closed circle, remove a neighborhood of that circle from both $R$ and $T$ and glue in two disjoint annuli consistent with the orientations of $R$ and $T$. Where $R$ and $T$ intersect in an arc, remove a neighborhood of that arc from both $R$ and $T$ and glue in two disjoint disks consistent with the orientations of $R$ and $T$. Call the resulting surface $R+T$.

Equivalently, we might define $V$ to be $\overline{\nu(R)}\cup\overline{\nu(T)}$ for small tubular neighborhoods of $R$ and $T$. Smooth $V$ at corners to be smoothly embedded in $M$. Then let $R+T$ denote the positive boundary component of $V$ (where the orientation on $V$ is induced by the orientations on $R$ and $T$).

Note $\chi(R+T)=\chi(R)+\chi(T)$.
\end{definition}

\begin{proposition}\label{cutpasteprop}
Let $\alpha_1$ and $\alpha_2$ be adjacent corners of a nondegenerate Thurston norm $x_M$, where $M$ is a $3$-manifold with boundary a disjoint union of tori $\sqcup_{i=1}^n P_i$. Then there exist norm-minimizing surfaces $S_1$ and $S_2$ with $[S_i]=\alpha_i$ so that for any positive integers $a$ and $b$, $aS_1+bS_2$ is properly norm-minimizing.
\end{proposition}

\begin{proof}
Since $x_M$ is nondegenerate, $\chi^+(R)=-\chi(R)$ for any norm-minimizing surface $R$ in $M$.

Let $S_1$ and $S_2$ be properly norm-minimizing surfaces with $[S_i]=\alpha_i$. Isotope the $S_i$ near their boundaries so that $\boundary S_1$ and $\boundary S_2$ intersect minimally. This ensures that for each $j$, every boundary component of $aS_1+bS_2$ on $P_j$ has the same orientation. We have $x_M(a[S_1]+b[S_2])=ax_M([S_1])+bx_M([S_2])=a\chi^+(S_1)+b\chi^+(S_2)=-a\chi(S_1)-b\chi(S_2)=-\chi(aS_1+bS_2)$. Then we are done if $aS_1+bS_2$ has no disk, $2$-sphere, torus or annulus components.

The condition on $\boundary(aS_1+bS_2)$ along with nondegeneracy of $x_M$ ensures that $aS_1+bS_2$ has do disk or annulus components. Suppose $S_1\setminus S_2$ includes a component $C$ which does not meet $\boundary S_1$ and with $\chi(C)\ge 0$. Then surger $S_2$ along $C$ to obtain $S'_2$ (i.e. $S'_2:=[S_2\setminus((\boundary C)\times I)]\cup(C\times S^0)$). See Figure~\ref{fig:simplecutpaste}. The surface $S'_2$ is homologous to $S_2$. Since $\chi(C)\ge 0$, $S'_2$ is norm-minimizing. Set $S_2:=S'_2$ and repeat until $S_1\setminus S_2$ includes no such component $C$. Now any closed components of $aS_1+bS_2$ much include a region homeomorphic to some component of $S_1\setminus S_2$, which must have negative Euler characteristic. Therefore, $aS_1+bS_2$ has no closed sphere or torus components, so $aS_1+bS_2$ is proeprly norm-minmizing.
\end{proof}

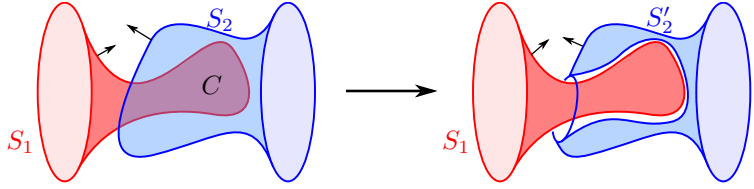
\begin{figure}
\begin{centering}
\labellist
\small\hair 2pt
\pinlabel $C$ at 215 120
\pinlabel $\textcolor{red}{S_1}$ at -20 50
\pinlabel $\textcolor{blue}{S_2}$ at 225 200
\pinlabel $\textcolor{red}{S_1}$ at 515 50
\pinlabel $\textcolor{blue}{S'_2}$ at 765 200
\endlabellist
\includegraphics[width=95mm]{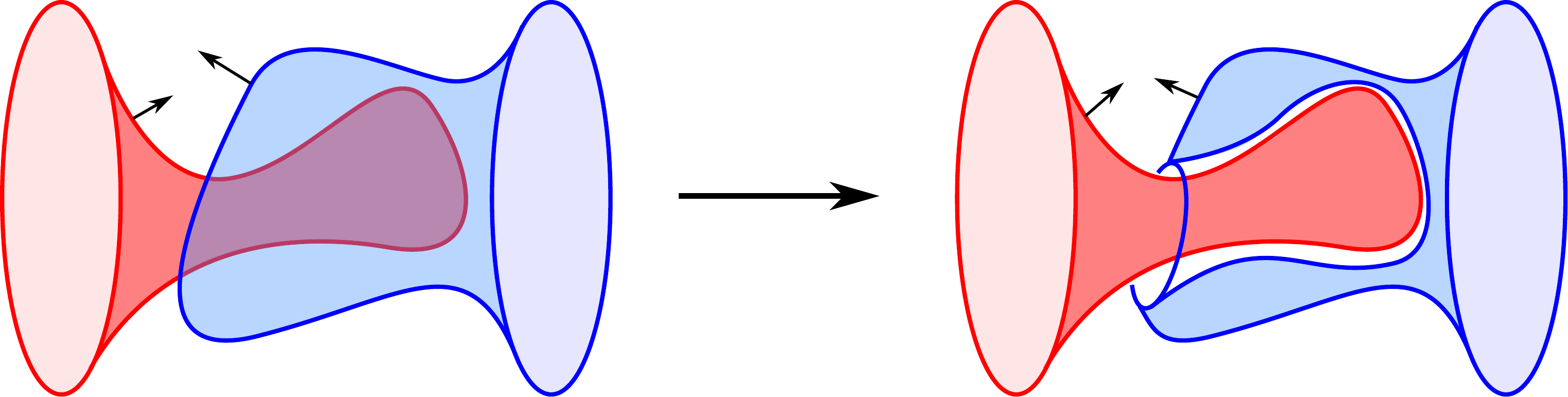}
\caption{{\bf{Left:}} $S_1\setminus S_2$ includes a disk $C$.  
{\bf{Right:}} We surger $S_2$ along $C$ to find another properly norm-minimizing surface $S'_2$. We replace $S_2:=S'_2$ and repeat until every component of $S_1\setminus S_2$ not meeting $\boundary S_1$ has negative Euler characteristic. Then $aS_1+bS_2$ cannot have any sphere or torus components.}\label{fig:simplecutpaste}\end{centering}
\end{figure}

There is a connection between the Thurston norm and foliations and sutured manifolds, which we dicuss in Section~\ref{sec:sutured}. Using these connections, Gabai~\cite{ft3m3} proved the following theorem.
\begin{theorem}[{\cite[Corollary 8.3]{ft3m3}}]\label{propr}
Let $S$ be a minimum-genus Seifert surface for a knot $K$ in $S^3$. Let $\widehat{S}$ be the closed surface in $S^3_0(K)$ obtained from $S$ by attaching a disk in the Dehn-surgery solid torus to $\boundary S$. Then $\widehat{S}$ is norm-minimizing.
\end{theorem}

Note that a minimum-genus Seifert surface for a knot $K$ is a properly norm-minimizing surface representing the generator of $H_2(S^3\setminus\nu(K),\boundary(S^3\setminus\nu(K));\R)$. The above theorem has the following important corollary, usually referred to as the Property R Conjecture.
\begin{corollary}[\cite{ft3m3}, Property R Conjecture]
If $K\subset S^3$ is a nontrivial knot, then $S^3_0(K)\not\cong S^1\times S^2$.
\end{corollary}
\begin{proof}
Let $S$ be a minimum-genus Seifert surface for $K$. Since $K$ is nontrivial, $S$ has positive-genus. By Theorem~\ref{propr}, $\widehat{S}$ is norm-minimizing. Then no $2$-sphere can represent the generating class $[\widehat{S}]$ of $H_2(S^3_0(K);\Z)$, so $S^3_0(K)\not\cong S^1\times S^2$.
\end{proof}

When $L$ is an $n$-component link in a rational homology $3$-sphere $Y$, $H_2(Y\setminus\nu(L),\boundary (Y\setminus\nu(L));\R)$ is $n$-dimensional, so we may consider surfaces more general than Seifert surfaces. The Thurston norm of a link complement should always be understood to mean the Thurston norm on relative homology.

For this discussion, fix $n=2$. One analogue of Theorem~\ref{propr} holds by Gabai and J. Rasmussen (independently of each other) but has not appeared in writing.

\begin{theorem}[Gabai, J. Rasmussen]\label{fillonethm}
Let $L=L_1\sqcup L_2$ be a $2$-component link in a rational homology $3$-sphere $Y$ with nonzero linking number. Assume $X:=Y\setminus\nu(L)$ has nondegenerate Thurston norm. Let $S$ be a norm-minimizing surface in $X$. Let $N$ denote the result of Dehn-filling $X$ at $P_1=\boundary\overline{\nu(L_1)}$ according to the slope of $\boundary S\cap P_1$. Let $\widehat{S}$ denote the surface in $N$ obtained from $S$ by capping off each component of $\boundary S\cap P_1$ by a disk in the Dehn-filling solid torus.

Assume $[S]$ is not a corner of the Thurston norm on $X$. Then $\widehat{S}$ is norm-minimizing.
\end{theorem}

Rasmussen's proof utilizes a connection between Thurston norm and knot Floer homology. We will present an argument of Gabai in Section~\ref{sec:sutured} via the theory of sutured manifolds, which is more similar to the content in this paper.

\begin{remark}In Theorem~\ref{fillonethm}, the condition that $[S]$ is not a corner of the Thurston norm is essential. For example, consider the link $L=L_1\sqcup L_2\subset S^3$ in Figure~\ref{fig:example1}. We see that if $S$ is the norm-minimizing surface in the third image (i.e. in the homology class of a punctured Seifert surface for $L_2$), then Dehn-filling $P_1$ according to the slope $\infty$ of $\boundary S\cap P_1$ yields $S^3\setminus\nu(L_2)$, in which $\widehat{S}$ is {\emph{not}} norm-minimizing. (In fact, $\widehat{S}$ is compressible.)
\end{remark}

Theorem~\ref{fillonethm} is similar in flavor to the following result of Baker and Taylor~\cite{baker}.

\begin{theorem}[{\cite[Theorem 4.6]{baker}}]\label{baker}
Let $X$ be a compact, connected, orientable, irreducible $3$-manifold whose boundary is a union of tori $P_1,\ldots, P_n$. Assume $X$ is not a cable space and $X\not\cong T^2\times I$. Let $S$ be a norm-minimizing surface in $X$. Let $N$ be the manifold obtained from $X$ by Dehn-filling $P_1$ according to the slope of $\boundary S\cap P_1$, and let $\widehat{S}\subset N$ be the surface obtained from $S$ by attaching disks to each component of $\boundary S\cap P_1$ inside the Dehn-filling solid torus.

There exists a finite set of slopes $E$ on $P_1$ so that when $\boundary S\cap P_1$ is not a slope in $E$, then $\widehat{S}$ is norm-minimizing.
\end{theorem}

Theorem~\ref{baker} applies to a more general class of $3$-manifolds than link complements in rational homology spheres, but Theorem~\ref{fillonethm} yields a stronger conclusion in its setting. In the same vein, we consider another reasonable analogue of Theorem~\ref{propr} in the setting of link complements.

\newtheorem*{2compthm}{Theorem~\ref{2compthm}}

\begin{2compthm}
Let $L=L_1\sqcup L_2$ be a $2$-component link in a rational homology $3$-sphere $Y$. Assume $\lk(L_1,L_2)\neq 0$ and that $X:=Y\setminus\nu(L)$ has nondegenerate Thurston norm. 

Let $S$ be a properly norm-minimizing surface in $X$. Let $\widehat{X}$ be the closed $3$-manifold obtained from $X$ by Dehn-filling both components $P_i=\boundary(\nu(L_i))$ of $\boundary X$ according to the slope $\boundary S\cap P_i$. Let $\widehat{S}$ be the closed surface in $\widehat{X}$ obtained by capping off each component of $\boundary S$ with a disk.

There exists a finite set $E\subset H_2(X,\boundary X;\mathbb{Z})$ so that if $[S]\not\in E$, then $\widehat{S}$ is norm-minimizing.

\end{2compthm}

\begin{remark}
Applying Theorem~\ref{baker} sequentially to each of $P_1$ and $P_2$ does not imply Theorem~\ref{2compthm}. 
When $L=L_1\sqcup\cdots\sqcup L_n$ is an $n$-component link, 
 Theorem~\ref{baker} is interesting when applied to $1,2,\ldots, n-1$ boundary components of $X$, but cannot constrain the number of elements of $H_2(X,\boundary X;\R)$ whose properly norm-minimizing representative $S$ fails to be norm minimizing when capped off in $Y_{\boundary S}(L)$ ($Y$ surgered along all $n$ components of $L$). We address the $n$-component link case in Theorem~\ref{ncompthm}.
\end{remark}

Theorem~\ref{2compthm} is a consequence of the following theorem.

\newtheorem*{mingenuscor}{Theorem~\ref{mingenuscor}}

\begin{mingenuscor}
Let $L=L_1\sqcup L_2$ be an $2$-component link in a rational homology $3$-sphere $Y$. Assume $\lk(L_1,L_2)\neq 0$ and that $X:=Y\setminus\nu(L)$ has nondegenerate Thurston norm. Let $S$ be a properly norm-minimizing surface in $X$ with $[S]$ primitive and not a corner of the Thurston norm. Let $\widehat{X}$ denote the closed manifold obtained from $X$ by Dehn filling each boundary component of $X$ according to the slope of $\boundary S\cap\boundary X$, and let $\widehat{S}\subset \widehat{X}$ be the closed surface obtained from $S$ by capping off each component of $\boundary S$ by a disk in a Dehn-filling solid torus. 

If $\widehat{S}$ is {\emph{not}} norm-minimizing, then either $g(S)=1$ or $g([S])$ is minimal among all classes in the interior of the same face of the Thurston norm as $[S]$. 
In particular, if $[S']$ is a primitive class in the interior of the same face as $[S]$ and $\widehat{S'}$ is also not norm-minimizing, then $g([S])=g([S'])$.
\end{mingenuscor}
By genus of a homology class, we mean $g(\alpha)$ to be the genus of a properly norm-minimizing surface representing $\alpha$. That is, in an $n$-component link complement, \[g(\alpha)=\frac{1}{2}\left(2+x_M(\alpha)-\sum_{i=1}^n|\langle\alpha, P_i\rangle|\right).\]

\begin{proof}[Proof of Theorem~\ref{2compthm} from Theorem~\ref{mingenuscor}]\label{proofof2compthm}
Assuming Theorem~\ref{mingenuscor}, to deduce Theorem~\ref{2compthm} we must show only that there are finitely many primitive elements of $H_2(X,\boundary X;\R)$ within each face of the Thurston norm on $X$ of any fixed genus. This fact follows from the following proposition.

\begin{proposition}\label{g0prop}
Fix an integer $g$. Let $\Sigma_{g}\subset H_2(X,\boundary X;\R)$ contain exactly the primitive elements of genus-${g}$. Then $\Sigma_{g}$ is finite.
\end{proposition}

\begin{proof}[Proof of Proposition~\ref{g0prop}]

Let $\lk=\lk(L_1,L_2)$. Let $\beta\in H_2(X,\boundary X;\R)$ be a norm-minimizing class with properly norm-minimizing representative $S$. We claim $|\boundary S|$ is bounded above uniformly (with bound depending only on $Y$ and $L$). 
Assuming this, then there is a uniform upper bound on $x([S])$ for $S\in\Sigma_{g}$. 
Nondegeneracy of Thurston norm then implies that $\Sigma_{g}$ is finite.

To see that $|\boundary S|$ is bounded above uniformly, 
say $[S]=p\alpha_1+q\alpha_2$,  
where $\alpha_i$ is the homology class of a punctured Seifert surface for $L_i$ and $p,q$ are coprime integers. Say $\boundary\alpha_i=m_i\gamma_i$ for some primitive $\gamma_i\in\boundary Y\setminus\nu(L)$ with $m_i>0$; that is, $m_i$ is the smallest positive number with $m_i[L_i]=0\in H_1(Y;\Z)$. Note $m_1\lk$ and $m_2\lk$ are integers, by definition of $\lk$.

Then $\boundary S$ meets $P_1$ in slope $-(qm_2\lk)/(pm_1)$ and $\boundary S$ meets $P_2$ in slope $-(pm_1\lk)/(qm_2)$. Then $\boundary S$ has $\gcd(-qm_2\lk,pm_1)$ components on $P_1$ and $\gcd(-pm_1\lk,qm_2)$ components on $P_2$. If $m_1=m_2=1$, then coprimeness of $p$ and $q$ immediately implies $|\boundary S|\le|\lk|+1$. But in general,

\begin{align*}
|\boundary S|&=\gcd(pm_1,qm_2\lk)+\gcd(pm_1\lk,qm_2)\\
&\le2|\lk|\gcd(pm_1,qm_2)\\
&\le2|\lk|\gcd(p,m_2)\gcd(q,m_1)\gcd(m_1,m_2)\\
&\le2|\lk|m_1^2m_2^2.
\end{align*}

\end{proof}

This completes the proof of Theorem~\ref{2compthm}.

\end{proof}

We prove Theorem~\ref{mingenuscor} in Section~\ref{sec:proof}.

Theorem~\ref{fillonethm} can be extended to $n$-component links.

\begin{theorem}\label{fillonethm2}
Let $L=L_1\sqcup \cdots\sqcup L_n$ be an $n$-component link in a rational homology $3$-sphere $Y$ with pairwise nonzero linking numbers. Fix an integer $k\in\{1,\ldots,n-1\}$. Assume $X:=Y\setminus\nu(L)$ has nondegenerate Thurston norm. Let $P_i=\boundary\overline{\nu(L_i)}$. Let $S$ be a properly norm-minimizing surface in $X$ which meets every boundary component of $X$, and let $N$ denote the result of Dehn-filling $X$ at $P_i$ according to the slope of $\boundary S\cap P_i$ for $i=1,\ldots, k$. Let $\widehat{S}$ denote the surface in $N$ obtained from $S$ by capping off each component of $\boundary S\cap P_i$ by a disk in the Dehn-filling solid torus, for $i=1,\ldots, k$.

Assume $[S]$ is not a corner of the Thurston norm on $X$. Then $\widehat{S}$ is norm-minimizing.
\end{theorem}
The proof is essentially the same as for Theorem~\ref{fillonethm}. We sketch both in Section~\ref{sec:sutured}.

Again, Theorem~\ref{fillonethm2} may be thought of as a version of Theorem~\ref{baker} restricted to a class of link complements, but yielding a stronger conclusion in that restricted setting. Similarly, Theorems~\ref{mingenuscor} and~\ref{2compthm} extend to $n$-component link complements.

\newtheorem*{mingenuscor2}{Theorem~\ref{mingenuscor2}}
\begin{mingenuscor2}
Let $L=L_1\sqcup\cdots\sqcup L_n$ be an $n$-component link in a rational homology $3$-sphere $Y$. Assume $\lk(L_i,L_j)\neq 0$ for each $i\neq j$. Let $X:=Y\setminus\nu(L)$. Assume $X$ has nondegenerate Thurston norm $x$.

Let $S$ be a properly norm-minimizing surface meeting every component of $\boundary X$ so that $[S]$ is primitive and is not a corner of $x$.

Let $\widehat{X}$ be the closed manifold obtained by Dehn-filling each boundary component of $X$ according to the slope of $\boundary S$. Let $\widehat{S}$ be the closed surface in $\widehat{X}$ obtained from $S$ by capping off each boundary component of $S$ by a disk in a Dehn-filling solid torus. Then at least one of the following is true:
\begin{itemize}
\item $g(S)=1$,
\item $\widehat{S}$ is norm-minimizing,
\item $g([S])\le g(\beta)$ whenever $\beta\in C$.
\end{itemize}
\end{mingenuscor2}

Theorem~\ref{mingenuscor2} is a generalization of Theorem~\ref{mingenuscor}. Note that when $n=2$, if $\beta,\alpha\neq 0\in H_2(X,\boundary X;\R)$ are not scalar multiples, then $\boundary\alpha,\boundary\beta$ have different slopes on each boundary component of $X$. The proof of Theorem~\ref{mingenuscor2} is exactly the same as the proof of Theorem~\ref{mingenuscor}, but we first prove Theorem~\ref{mingenuscor} separately as the statement and notation are simpler. We prove Theorem~\ref{mingenuscor2} in Section~\ref{sec:ncomp}.

\newtheorem*{ncompthm}{Theorem~\ref{ncompthm}}

\begin{ncompthm}
Let $L=L_1\sqcup\cdots\sqcup L_n$ be an $n$-component link $(n>1)$ in a rational homology $3$-sphere $Y$ with pairwise nonzero linking numbers. Assume $X:=Y\setminus\nu(L)$ has nondegenerate Thurston norm.

Let $S$ be a properly norm-minimizing surface in $X$ meeting every component of $\boundary X$ . Let $\widehat{X}$ be the closed manifold obtained from $X$ by Dehn filling $X$ according to $\boundary S$, and let $\widehat{S}\subset\widehat{X}$ be the closed surface obtained from capping off each boundary component of $S$ within the Dehn-filling solid tori.

Let $\widetilde{Y}$ be the $3$-manifold obtained from $Y$ by surgering $Y$ along $L_3\sqcup\cdots\sqcup L_n$ according to $\boundary S$.

There exists an $(n-2)$-dimensional set of rays $E$ from the origin of $H_2(X,\boundary X;\R)\cong\R^n$ so that if $[S]\not\in E$, then either $\widehat{S}$ is norm-minimizing or $\widetilde{Y}\setminus\nu(L_1\sqcup L_2)$ has degenerate Thurston norm.
\end{ncompthm}

In Section~\ref{sec:ncomp}, we will prove Theorem~\ref{ncompthm} by induction. Theorem~\ref{2compthm} will be the base case of this induction.

\section{Sutured manifolds and foliations}\label{sec:sutured}

In this section we give some necessary background on sutured manifolds, introduced by Dave Gabai and explored in many papers including~\cite{ft3m1},\cite{ft3m2}, and \cite{ft3m3}. We also give some background on foliations; some excellent further resources are by Calegari~\cite{calegari} and Novikov~\cite{novikov}.

  In the author's mind, the language of sutured manifolds and foliations are often interchangeable. The main arguments of this paper could likely be rewritten completely in terms of foliations while avoiding sutured manifolds entirely or vice versa, but would be much more cumbersome.

  \subsection{Definitions and important theorems}

\begin{definition}[{\cite[Definition 2.6]{ft3m1}}]
A sutured manifold $(M,\gamma)$ is a compact, oriented $3$-manifold $M$ together with a set $\gamma$ of pairwise disjoint annuli $A(\gamma)$ and tori $T(\gamma)$ in $\boundary M$. We may abuse notation and view $\gamma$ as a subset of $\boundary M$. One oriented core of each element of $A(\gamma)$ is called a {\emph{suture}}. The set of all such sutures is denoted by $s(\gamma)$.

Every component of $R(\gamma):=\boundary M\setminus\int{\gamma}$ is oriented. These orientations are taken to be coherent with respect to $s(\gamma)$, so if a component $R$ of $R(\gamma)$ has one boundary component in some $A\in A(\gamma)$, then the induced orientation on $\boundary R$ must agree with the orientation of the suture in $A$.
 Let $R_+(\gamma)\subset R(\gamma)$ include all components whose whose normal vectors point out of $M$. Let $R_-(\gamma)=R(\gamma)\setminus R_+(\gamma)$ include all components of $R(\gamma)$ whose normal vectors point into $M$.
\end{definition}

We will be primarily interested in the case that $(M,\gamma)$ is the {\emph{complementary sutured manifold}} to a surface.

\begin{definition}
Let $S$ be an oriented surface properly embedded in a compact $3$-manifold $X$ whose boundary is a (possible empty) union of tori.  Let $V=S\times I$, where the $I$ direction is chosen so that for each boundary component $C$ of $S$, $C\times I\subset V$ is contained in $\boundary X$. Orient the $I$ direction of $S\times I$ so that the normal vector to $S\times 1$ points into $V$ and the normal vector to $S\times0$ points out of $V$.

The {\emph{complementary sutured manifold}} $(M,\gamma)$ to $S$ is defined by $M=\overline{X\setminus V}$, and the elements of $\gamma$ are the components of $\overline{(\boundary X)\setminus V}$. The orientations of each suture are chosen so that $R_+(\gamma)=S\times 1$ and $R_-(\gamma)=S\times0$ (hence our choice of orientation of the $I$ direction of $S\times I$).

Note $|A(\gamma)|=|\boundary S|$. Moreover, if $\boundary S$ meets every component of $\boundary X$, then $T(\gamma)=\emptyset$. 
\end{definition}

See Figure~\ref{fig:complementarymanifold} for an example of a complementary sutured manifold to a Seifert surface in a knot complement.

\begin{figure}
\begin{centering}
\labellist
\small\hair 2pt
\pinlabel $R_+(\gamma)$ at 260 55
\pinlabel $R_-(\gamma)$ at 250 115
\endlabellist
\includegraphics[width=80mm]{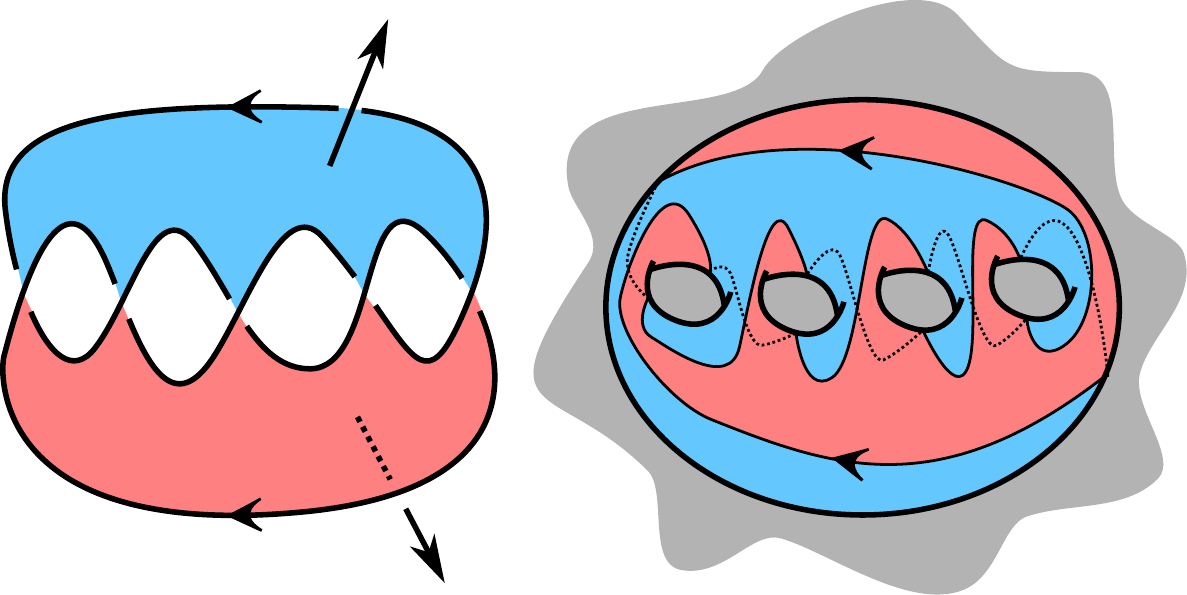}
\caption{{\bf{Left:}} a Seifert surface $S$ in $S^3\setminus\nu(K)$. {\bf{Right:}} The complementary sutured manifold $(M,\gamma)$. In the picture we draw a genus-$4$ surface $\boundary M$ in $\R^3$. The manifold $M$ is the unbounded region in $\R^3$, compactified by $\R^3\subset S^3$. The elements of $A(\gamma)$ are drawn as thin black annuli, with orientations of the sutures indicated by arrows. Note that the normal vector to $R_+(\gamma)$ points out of $M$, while the normal vector to $R_-(\gamma)$ points into $M$.}\label{fig:complementarymanifold}\end{centering}
\end{figure}

To relate sutured manifolds to the Thurston norm, we need the following key definition.

\begin{definition}\label{tautdef}
A sutured manifold $(M,\gamma)$ is {\emph{taut}} if $M$ is irreducible and each of $R_+(\gamma), R_-(\gamma)$ is norm-minimizing in $H_2(M,\gamma;\R)$.
\end{definition}

\begin{remark}\label{tautremark}
From Definition~\ref{tautdef}, if $S$ is a surface properly embedded in a compact irreducible manifold $X$ with $\boundary X$ a collection of tori, then the complementary sutured manifold $(M,\gamma)$ to $S$ is taut if and only if $S$ is norm-minimizing in $H_2(X,\boundary X;\R)$.
\end{remark}

Thus, to show a surface is norm-minimizing, it is sufficient to show that a certain sutured manifold is taut. We can show tautness using foliations.

\begin{definition}
A codimension-1 oriented {\emph{foliation}} $\F$ of a compact $3$-manifold $X$ is a collection of oriented (perhaps noncompact) connected surfaces $\{F_\alpha\}_{\alpha\in\Lambda}$ called the {\emph{leaves}} of $\F$ so that:
\begin{itemize}
\item $X=\cup_\Lambda F_\alpha$ where $F_\alpha$ is smoothly embedded in $X$ and the leaves $F_\alpha,F_\beta$ are disjoint when $\alpha\neq\beta$,
\item Every point in the interior of $X$ has a neighborhood $U$ and a coordinate map $(x_1,x_2,x_3):U\to \R^3$ so that the components of $F_\alpha\cap U$ are defined by $x_3=c$ for a constant $c$, all oriented so the positive normal vector points in the direction of $(0,0,1)$,
\item If a component $T$ of $\boundary X$ is not a leaf of $\F$, then every point in $T$ has a neighborhood $U$ and a coordinate map $(x_1,x_2,x_3):U\to \{(x,y,z)\mid x\ge 0\}\subset\R^3$ so that the components of $F_\alpha\cap U$ are defined by $x_3=c$ for a constant $c$, all oriented so the positive normal vector points in the direction of $(0,0,1)$. 
\end{itemize}
\end{definition}

Note that a leaf of a foliation $\F$ of a compact $3$-manifold $X$ meets $\boundary X$ only transversely or is an entire boundary component of $X$. When adapting foliations to the setting of sutured manifolds, this condition changes.

\begin{definition}
A codimension-1 oriented {\emph{foliation}} $\F$ of sutured manifold $(M,\gamma)$ is a collection of oriented (perhaps noncompact) connected surfaces $\{F_\alpha\}_{\alpha\in\Lambda}$ called the {\emph{leaves}} of $\F$ so that
\begin{itemize}
\item $M=\cup_\Lambda F_\alpha$ where $F_\alpha$ is smoothly embedded in $M$ and the leaves $F_\alpha,F_\beta$ are disjoint when $\alpha\neq\beta$,
\item Every point in $M$ has a neighborhood $U$ and a coordinate map $(x_1,x_2,x_3):U\to \R^3$ so that the components of $F_\alpha\cap U$ are defined by $x_3=c$ for a constant $c$, all oriented so the positive normal vector points in the direction of $(0,0,1)$,
\item Every point in the interior of $A(\gamma)$ or a component of $T(\gamma)$ which is not a leaf of $\F$ has a neighborhood $U$ and a coordinate map $(x_1,x_2,x_3):U\to \{(x,y,z)\mid x\ge 0\}\subset\R^3$ so that the components of $F_\alpha\cap U$ are defined by $x_3=c$ for a constant $c$, all oriented so the positive normal vector points in the direction of $(0,0,1)$,
\item Each component of $R(\gamma)$ is a leaf of $\F$ (with the same orientation as $R(\gamma)$).
\end{itemize}
\end{definition}

Thus, the leaves of a foliation $\F$ on a sutured $3$-manifold meet $A(\gamma)$ and perhaps some elements of $T(\gamma)$ transversely. This part of $\boundary M$ is often called the {\emph{vertical}} boundary of $M$, as we imagine the leaves of $\F$ as being horizontal. In contrast, $R_+(\gamma)$ and $R_-(\gamma)$ (and perhaps some elements of $T(\gamma)$) are leaves of $\F$. This part of $\boundary M$ is often called the {\emph{horizontal}} boundary of $M$.

Foliations have been studied in more general settings; see e.g.~\cite{lawson} for more exposition (in non-sutured manifolds). From now on we will always use the term foliation to refer to a codimension-1 oriented foliation of a compact $3$-manifold or sutured manifold -- the context will always be clear.

\begin{definition}\label{tautfoldef}
A foliation $\F$ on a compact $3$-manifold $M$ or a sutured manifold $(M,\gamma)$ is said to be {\emph{taut}} if there exists a circle or properly embedded path $\eta$ in $M$ which intersects every leaf of $\F$ and always intersects leaves of $\F$ transversely.
\end{definition}

Definition~\ref{tautfoldef} allows one to relate foliations and the Thurston norm.

\begin{theorem}[{\cite[Corollary 2]{thurston}}]\label{studyfoliation}
Let $\F$ be a taut foliation on a compact $3$-manifold $M$. Suppose $L$ is a leaf of $\F$ which is a compact surface. Then $L$ is norm-minimizing in the class $[L]\in H_2(M,\boundary M;\R)$.
\end{theorem}

Gabai proved the converse to Theorem~\ref{studyfoliation} in the following setting.

\begin{theorem}[{\cite[Theorem 5.5]{ft3m1}}]\label{thm5.5}\label{havefoliation1}
Let $M$ be a compact connected irreducible oriented $3$-manifold $M$ whose boundary is a (possibly empty) union of tori. Let $S$ be a properly norm-minimizing surface representing a nontrivial class of $H_2(M,\boundary M;\R)$. Then $S$ is a leaf of a taut foliation $\F$ of $M$ which has no Reeb components in $\boundary M$.
\end{theorem}

We did not previously define Reeb components. For our purposes, we can use the following nonstandard definition.

\begin{definition}
Let $M$ be a compact connected irreducible oriented $3$-manifold $M$ whose boundary is a (possibly empty) union of tori $P_1\sqcup\cdots\sqcup P_n$. Let $S$ be a properly norm-minimizing surface in $M$. Suppose $S$ is a leaf of a foliation $\F$. We say $\F$ has no Reeb components in $P_i$ if any curve on $P_i$ not of the same slope as $\boundary S\cap P_i$ can be perturbed to be transverse to $\F|_{P_i}$. If $\boundary M=\emptyset$ or if $\F$ has no Reeb components in any $P_i$ for $i=1,\ldots, n$, then we say $\F$ has no Reeb components in $\boundary M$.
\end{definition}

The general strategy in the proof of Theorem~\ref{2compthm} is to construct a taut foliation $\G$ on $\widehat{X}$ achieving $\widehat{S}$ as a compact leaf. We may then conclude by Theorem~\ref{studyfoliation} that $\widehat{S}$ is norm-minimizing. To construct this taut foliation, we begin with a taut foliation $\F=\{F_\alpha\}$ on $X$ achieving $S$ as a leaf. The existence of $\F$ is ensured by Theorem~\ref{havefoliation1}. (Note here that we actually need $X=Y\setminus\nu(L)$ to be irreducible to apply Theorem~\ref{havefoliation1} -- but if $X$ decomposes as a connect-sum, then we will apply Theorem~\ref{havefoliation1} to a prime summand of $X$.) If it happens to be the case that each $F_\alpha|_{\boundary X}$ consists of compact circles, then we may Dehn-fill $\boundary X$ and cap off each component of $F_\alpha|_{\boundary X}$ with a disk to construct a taut foliation on $\widehat{X}$. However, in general we cannot hope that $F_\alpha|_{\boundary X}$ is compact, so we must first attempt to alter $\F$. In Section~\ref{sec:operations}, we describe tools introduced by Gabai~\cite{suspensions} to change a taut foliation $\F$.

\subsection{Sutured manifold and foliation operations}\label{sec:operations}

The most basic operations one can perform on a sutured manifold $(M,\gamma)$ (and the only ones needed in this paper) are  {\emph{product-disk decomposition}} and {\emph{product-annulus decomposition}}.

\begin{definition}[\cite{generaoflinks}]
Let $(M,\gamma)$ be a sutured manifold.

A {\emph{product disk}} in $(M,\gamma)$ is a disk $D$ properly embedded in $M$ so that for some coordinates $D=I\times I$, we have $I\times\boundary I\subset\gamma$, $1\times I\subset R_+(\gamma)$, and $0\times I\subset R_-(\gamma)$.

A {\emph{product annulus}} in $(M,\gamma)$ is an annulus $A$ properly embedded in $M$ so that $\boundary A$ has two components $\boundary_+ A\subset R_+(\gamma)$ and $\boundary_- A\subset R_-(\gamma)$.

\end{definition}

We draw a schematic of a product disk in Figure~\ref{fig:productdisk}.

\begin{figure}
\labellist
\small\hair 2pt
\pinlabel $R_+(\gamma)$ at 160 220
\pinlabel $R_-(\gamma)$ at 140 -20
\pinlabel $D$ at 150 100
\endlabellist
\begin{centering}
\includegraphics[width=100mm]{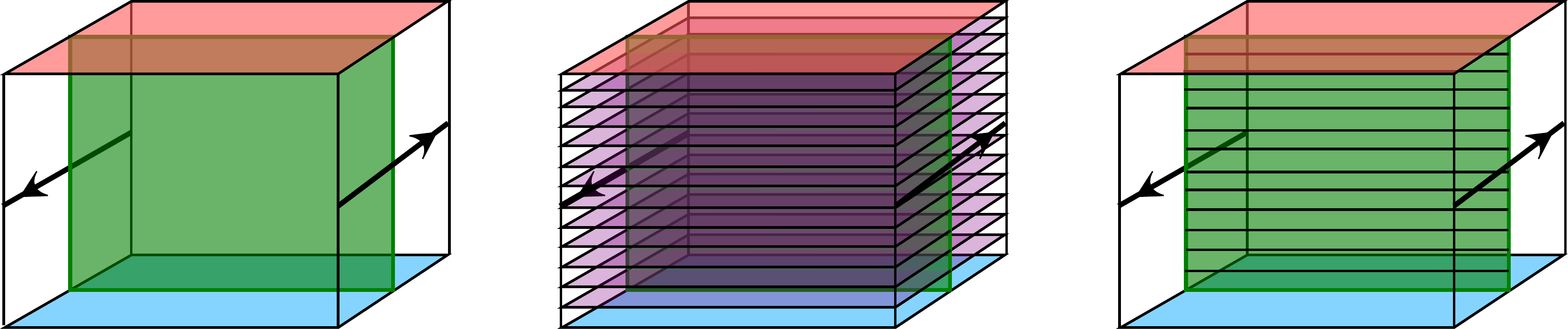}
\caption{{\bf{Left:}} a product disk $D$ in a sutured manifold $(M,\gamma)$. {\bf{Middle:}} part of a taut foliation $\F$ on $(M,\gamma)$. {\bf{Right:}} Up to perturbation of $D$, $\F|_D$ is a product foliation.}\label{fig:productdisk}\end{centering}
\end{figure}

\begin{remark}\label{transverseremark}
Let $\F$ be a taut foliation on a sutured manifold $(M,\gamma)$ containing a product disk or annulus $\Delta$. By work of Thurston~\cite{thurston} and Roussarie~\cite{roussarie}, $\Delta$ can be perturbed so that $\int{\Delta}$ intersects the leaves of $\F$ transversely.

Given a sutured manifold with taut foliation $\F$, we will always assume that product disks and annuli are transverse to $\F$.
\end{remark}

\begin{definition}[{\cite[Definition 3.6]{generaoflinks},\cite[Definition 3.1]{ft3m1}}]
Let $\Delta$ be a product disk or annulus inside a sutured manifold $(M,\gamma)$. Let $V=\Delta\times I$, where the $I$  direction is chosen so that $(\boundary \Delta)\times I\subset \boundary M$. {\emph{Sutured manifold decomposition}} (or more simply, {\emph{decomposition}}) of $(M,\gamma)$ along $\Delta$ yields the sutured manifold $(M',\gamma')$, where $M'=\overline{M\setminus V}$. If $\Delta$ is a disk, then $\gamma'$ is obtained from $\gamma$ by surgering the element(s) of $A(\gamma)$ which meet $\Delta$ along $\Delta$ (see Figure~\ref{fig:diskdecomp} or Figure~\ref{fig:diskdecomplemma}). If $\Delta$ is an annulus, then $\gamma'$ is a superset of $\gamma$ but also includes the two annuli $\Delta\times 0,\Delta\times 1$. The sutures in these annuli are oriented consistently with the orientations on $R(\gamma)$.

We write $(M,\gamma)\xrightarrow\Delta(M',\gamma')$, or say $(M',\gamma')$ is the result of (product) disk/annulus decomposition of $(M,\gamma)$ along $\Delta$. We draw an example of product disk decomposition in Figure~\ref{fig:diskdecomp}.
\end{definition}

\begin{figure}
\labellist
\small\hair 2pt
\pinlabel $R_+(\gamma)$ at 75 30
\pinlabel $R_-(\gamma)$ at 60 92
\pinlabel $R_+(\gamma')$ at 255 30
\pinlabel $R_-(\gamma')$ at 240 92
\endlabellist
\begin{centering}
\includegraphics[width=90mm]{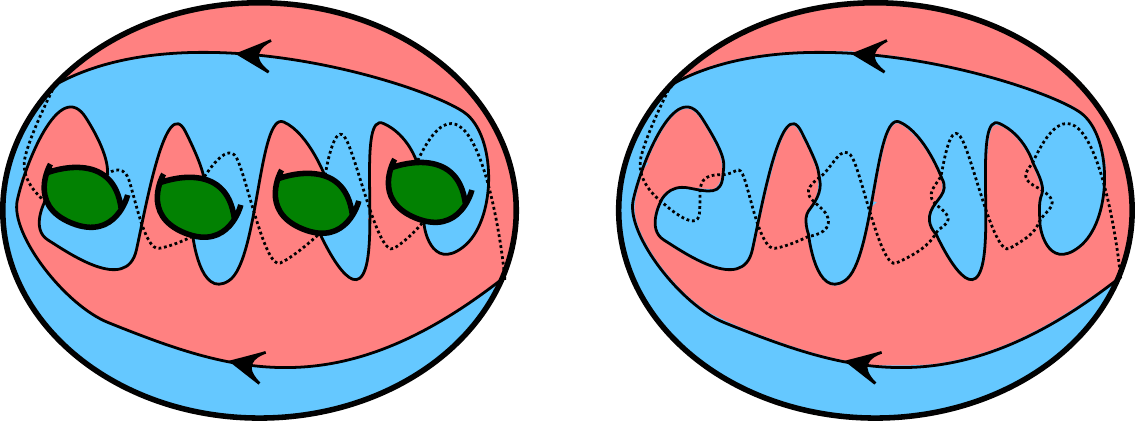}
\caption{{\bf{Left:}} a sutured manifold $(M,\gamma)$. In the picture we draw a genus-$4$ surface $\boundary M$ in $\R^3$. The manifold $M$ is the unbounded region in $\R^3$, compactified by $\R^3\subset S^3$. The elements of $A(\gamma)$ are drawn as thin black annuli, with orientations of the sutures indicated by arrows. The sutured manifold $(M,\gamma)$ is the complementary manifold to a Seifert surface $S$ in $S^3$, as shown in Figure~\ref{fig:complementarymanifold}. We shade four product disks in $(M,\gamma)$. {\bf{Right:}} We decompose $(M,\gamma)$ along each of the four product disks. The result is a sutured manifold $(M',\gamma')$ with $M'\cong B^3$ and $\gamma'$ a single annulus. This sutured manifold is taut, so we conclude by Theorem~\ref{tautdecomp} that $(M,\gamma)$ is taut. This implies that $S$ is a norm-minimizing  surface.}\label{fig:diskdecomp}\end{centering}
\end{figure}
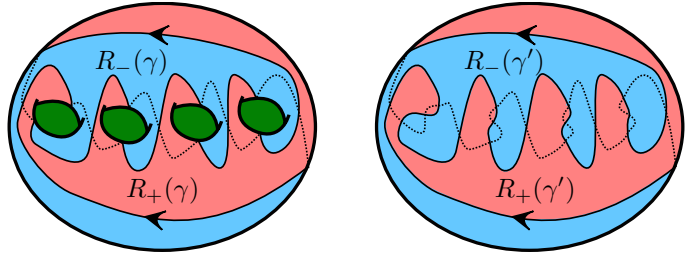

\begin{remark}\label{tautdecompfol}
Let $\F$ be a taut foliation on sutured manifold $(M,\gamma)$. Let $\Delta$ be a product disk or annulus in $(M,\gamma)$. By Remark~\ref{transverseremark}, $\F':=\F|_{M'}$ is a foliation on $(M',\gamma')$, where $(M,\gamma)\xrightarrow{\Delta}(M',\gamma')$. In fact, $\F'$ is taut. (If $M'=M'_1\sqcup M'_2$ is disconnected, then we mean $\F'|_{M'_1}$ and $\F'_{M'_2}$ are both taut.)
\end{remark}

\begin{lemma}[{\cite[Lemma 3.12]{ft3m1}}]\label{tautdecomp}
Let $(M,\gamma)\xrightarrow{\Delta}(M',\gamma')$ be a product disk or annulus decomposition. Then $(M,\gamma)$ is taut if and only if $(M',\gamma')$ is taut.
\end{lemma}

\begin{remark}
By Remark~\ref{transverseremark}, if $(M,\gamma)$ admits a taut foliation and $(M,\gamma)\xrightarrow{\Delta}(M',\gamma')$ is a product disk or annulus decomposition, then $(M',\gamma')$ admits a taut foliation.
\end{remark}

From Lemma~\ref{tautdecomp}, we obtain the key fact in the proof of Theorems~\ref{fillonethm} and~\ref{fillonethm2}.

\begin{lemma}[Gabai]\label{lemmadiskdecomp}
Let $(M,\gamma)$ be a sutured manifold. Suppose a product disk $D$ connects two distinct elements $A_1$ and $A_2$ of $A(\gamma)$. Let $(M_1,\gamma_1)$ be the sutured manifold obtained from $(M,\gamma)$ by surgering $M$ at $A_1$ (i.e. attaching a $3$-dimensional $2$-handle to $M$ along $A_1$) and then forgetting the suture $A_1$ (i.e. $M_1=M\cup_{A_1}E\times I$ for a disk $E$ with boundary the core of $A_1$, $\gamma_1=\gamma\setminus A_1$). Let $(M_2,\gamma_2)$ be the sutured manifold obtained from $(M,\gamma)$ by product-disk decomposition at $D$. Then $(M_1,\gamma_1)\cong (M_2,\gamma_2)$.

\end{lemma}

\begin{proof}
Let $B\subset M_1$ be given by $B:=\overline{(E\times I)\cup (I\times D)}$, so $B$ is a $3$-ball meeting $A_2$ in a disk. 
Deform $\nu(\boundary M_1)$ near $B$, sweeping this disk through $B$. 
This realizes a map $M_1\to M_2,\gamma_1\to\gamma_2$. See Figure~\ref{fig:diskdecomplemma}.

\end{proof}

\begin{figure}
\labellist
\small\hair 2pt
\pinlabel $A_1$ at 735 250
\pinlabel $A_2$ at 1075 250
\pinlabel $D$ at 910 410
\pinlabel $E$ at 735 360
\endlabellist
\begin{centering}
\includegraphics[width=100mm]{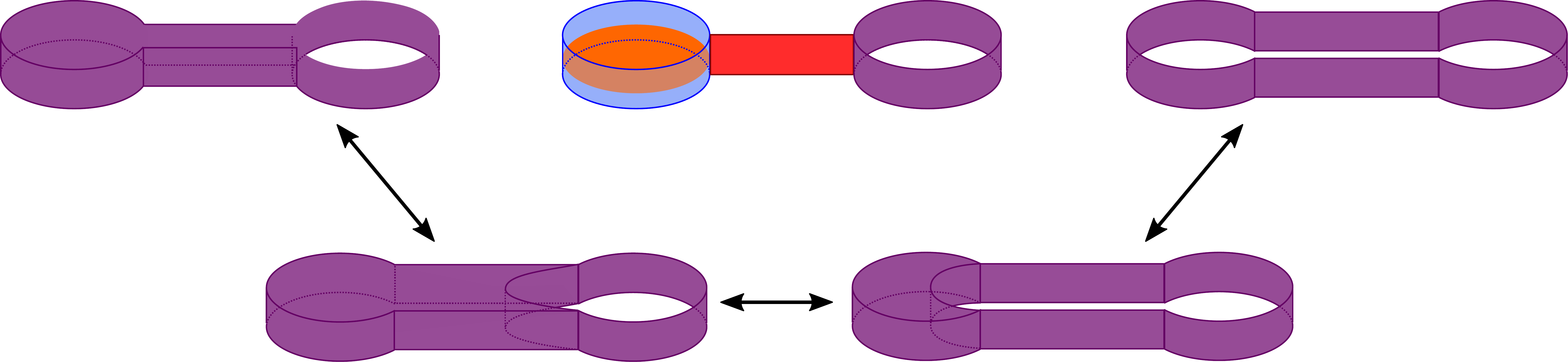}
\caption{{\bf{Top Middle:}} Two distinct elements $A_1, A_2$ of $A(\gamma)$ in a sutured manifold $(M,\gamma)$. We indicate a product disk $D$ connecting $A_1$ and $A_2$. We draw a disk $E$ with boundary the core of $A_1$ to indicate the core of a $2$-handle $E\times I$ that we might attach to $M$. {\bf{Top Left:}} The sutured manifold $(M_1,\gamma_1)$, where $M_1=M\cup(E\times I)$ and $\gamma_1=\gamma\setminus A_1$. {\bf{Top Right:}} The sutured manifold $(M_2,\gamma_2)$, which is obtained from $(M,\gamma)$ by product disk decomposition along $D$. {\bf{Bottom:}} The sutured manifolds $(M_1,\gamma_1)$ and $(M_2,\gamma_2)$ are homeomorphic.}\label{fig:diskdecomplemma}\end{centering}
\end{figure}
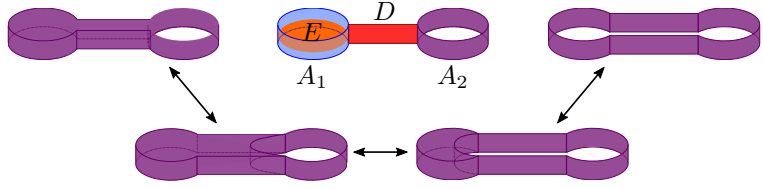

We are now able to prove Theorems~\ref{fillonethm} and~\ref{fillonethm2} via an argument of Gabai.

\begin{proof}[Proof of Theorems~\ref{fillonethm} and~\ref{fillonethm2}; Gabai, J. Rasmussen]

We will first prove Theorem~\ref{fillonethm}. Let $P_2=\boundary\overline{\nu(L_2)}$. 
Suppose $S'$ is a properly norm-minimizing surface with $[S]=[S']$. Then $g(S)\le g(S')$ and $|\boundary S\cap P_2|\ge |\boundary S'\cap P_2|$, so $\chi(\widehat{S})\ge\chi(\widehat{S'})$. Therefore, it would be sufficient to prove the claim for properly norm-minimizing surfaces, so assume $S$ is properly norm-minimizing.

Since $[S]$ is not a corner of the Thurston norm, we have $c[S]=a[R]+b[T]$ for some positive integers $a,b$, and $c$ and adjacent corners $[R]$ and $[T]$. Take $R$ and $T$ to be properly norm-mimizing surfaces so that $aR+bT$ is a properly norm-minimizing surface, using Proposition~\ref{cutpasteprop}. Then $\chi(\widehat{aR+bT})=c\chi(\widehat{S})$. If $\widehat{aR+bT}$ is norm-minimizing, so is $\widehat{S}$, so to prove the claim we may assume that $S=aR+bT$.

 Let $(M,\gamma)$ be the complementary sutured manifold to $S$ in $X$. Since $S$ is norm-minimizing, $(M,\gamma)$ is taut. For each intersection of $aS_1$ with $bS_2$, we obtain a product disk in $(M,\gamma)$. See Figure~\ref{fig:cutandpaste}.

\begin{figure}
\labellist
\small\hair 2pt
\pinlabel $\textcolor{red}{R}$ at -20 80
\pinlabel $\textcolor{blue}{T}$ at 150 180
\pinlabel $R+T$ at 460 180
\endlabellist
\begin{centering}{\includegraphics[width=100mm]{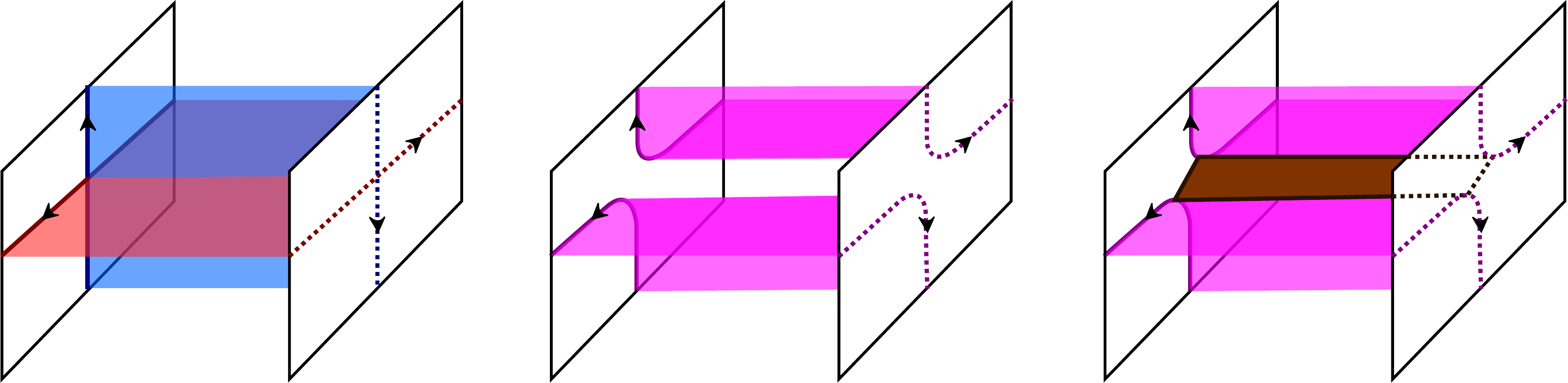}}
\caption{{\bf{Left:}} An arc of intersection between two surfaces $R$ and $T$. {\bf{Middle:}} The cut-and-paste surface $S=R+T$. {\bf{Right:}} A product disk in the complementary sutured manifold to $S$.}\label{fig:cutandpaste}
\end{centering}
\end{figure}

Since $\lk:=\lk(L_1,L_2)\neq 0$, the map $\boundary:H_2(X,\boundary X;\R)\to H_1(\boundary X;\R)$ is an injection. Thus $\boundary R$ and $\boundary T$ have nonmultiple slopes on $P_1$, so $\boundary R\cap\boundary T\neq\emptyset$. Then $R$ and $T$ must have some arcs of intersection.

Take $\boundary R$ and $\boundary T$ to intersect minimally so that any arc of intersection between $R$ and $T$ meets both $P_1$ and $P_2$. Then for each element $A$ of $A(\gamma)$ in $P_1$, there is a product disk connecting $A$ to an element of $A(\gamma)$ in $P_2$.

Let $(M',\gamma')$ be the complementary sutured manifold to $\widehat{S}$ in $N$. Then $M'$ is obtained from $M$ by attaching $3$-dimensional $2$-handles to each element of $A(\gamma)$ in $P_1$, and $\gamma'$ is obtained from $\gamma$ by forgetting those elements of $A(\gamma)$. By Lemma~\ref{lemmadiskdecomp}, $(M',\gamma')$ is the result (up to homeomorphism) of a sequence of product-disk decompositions of $(M,\gamma)$. Therefore, tautness of $(M,\gamma)$ implies that $(M',\gamma')$ is taut, so $\widehat{S}$ is norm-minimizing. This concludes the proof of Theorem~\ref{fillonethm}.

Now we prove Theorem~\ref{fillonethm2}. It is again sufficient to prove the claim when $S$ is properly norm-minimizing. Since $[S]$ is not a corner of the Thurston norm on $X$, we can write $c[S]=a[R]+b[T]$ for positive integers $a,b,c$, where $R$ and $T$ are properly norm-minimizing surfaces and $[R], [T]$ are adjacent corners. There are now many choices of $[R]$ and $[T]$. Again, it is sufficient to consider $S=aR+bT$ (that is, $c=1$).

For appropriate choice of $[R]$ and $[T]$, $\boundary R$ and $\boundary T$ both meet each boundary component of $X$ and do so in distinct slopes. Here we are using the pairwise nonzero linking numbers of $L_1,\ldots, L_n$ -- this means that for fixed $q\in\Q$, $j\in\{1,\ldots, n\}$, there is an $(n-2)$-dimensional set of rays in $H_2(X,\boundary X;\R)$ with boundary slope $q$ on $P_j$ There is an $(n-3)$-dimensional set of rays in $H_2(X,\boundary X;\R)$ with no boundary on $P_i$, because there are $(n-1)$-dimensions of rays in $H_2(X,\boundary X;\R)$ altogether and the condition on $P_i$ gives two $1$-dimensional conditions -- as a sum of homology classes of punctured Seifert surfaces for $L_j$'s, there cannot be any for $L_i$ and there is a linear equation on the other summands in terms of the linking numbers that must be satisfied). Let $\{A_1,\ldots, A_m\}$ be the elements of $A(\gamma)$ in $P_1\sqcup\cdots\sqcup P_k$. Then for each $j$, there is a product disk corresponding to an intersection of $aR$ with $bT$ meeting $A_j$. By concatenating product disks, we can find disjoint product disks $D_1,\ldots, D_m$ so that $D_i$ meets $A_i$ and an element of $A(\gamma)$ in $P_n$.

Let $(M',\gamma')$ be the complementary sutured manifold to $\widehat{S}$ in $N$. Then $M'$ is obtained from $M$ by attaching $3$-dimensional $2$-handles to each element of $A(\gamma)$ in $P_1\sqcup\cdots\sqcup P_{k}$, and $\gamma'$ is obtained from $\gamma$ by forgetting those elements of $A(\gamma)$. By Lemma~\ref{lemmadiskdecomp}, $(M',\gamma')$ is the result (up to homeomorphism) of a sequence of product-disk decompositions of $(M,\gamma)$. Therefore, tautness of $(M,\gamma)$ implies that $(M',\gamma')$ is taut, so $\widehat{S}$ is norm-minimizing.

\end{proof}

As well as modifying sutured manifolds, one can use product disks and annuli to modify foliations. In particular, if $\F$ is a taut foliation on $M$, then we are interested in $\F|_{\boundary M}$.

\begin{definition}
Let $M$ be a compact $3$-manifold with some torus boundary component $P$. Let $\F$ be a foliation on $M$ with no Reeb components on $P$. Suppose $\F|_P$ includes two circles $C_1, C_2$ cobounding annulus $A\subset P$. Choose coordinates $A=I\times S^1=I\times[0,2\pi]/\sim$. Because $\F|_P$ does not have a Reeb component in $P$, there is some boundary-preserving automorphism $f:I\to I$ so that $\lim_{t\to 0^+}x\times t$ and $\lim_{t\to 2\pi^-}f(x)\times t$  are contained in the same leaf of $\F$ -- i.e. $\F|_A$ is induced by the mapping torus structure $A=I\times[0,2\pi]/((x,0)\sim(f(x),1))$. We say that $\F|_A$ is a {\emph{suspension}} of $f$.
\end{definition}

We summarize a few operations introduced by Gabai~\cite{suspensions} which we use in the proof of Theorem~\ref{2compthm}. We skip several interesting operations which we will not make use of, and state others in less than full generality.

\begin{operation}[Thickening Leaves,~\cite{denjoy},{\cite[Operation 2.1.1]{suspensions}}]
Suppose $L$ is a leaf of a taut foliation $\F$ on a $3$-manifold $M$ or sutured manifold $(M,\gamma)$. Assume $L$ is two-sided in $M$. Let $M'$ be obtained from $M$ by deleting $L$ and replacing it with $L\times I$. (If $M$ is sutured, then extend $\gamma$ to $\boundary M'$ naturally, with $(\boundary L)\times I$ contained in elements of $A(\gamma)$.) Let $\F'$ be the taut foliation of $M'$ which agrees with $\F$ outside of $L\times I$, and includes $L\times t$ as a leaf for each $t\in I$.

Identify $M'$ with $M$. We say the taut foliation $\F'$ on $M$ is obtained from $\F$ by {\emph{thickening}} the leaf $L$.
\end{operation}

\begin{operation}[$I$-bundle replacement, {\cite[Operation 2.1.3]{suspensions}}]
Let $L$ be a leaf of a taut foliation $\F$ on a $3$-manifold $M$ or sutured manifold $(M,\gamma)$. Assume $L$ is two-sided in $M$. Let $M'$ be obtained from $M$ by deleting $L$ and replacing it with $L\times I$. (If $M$ is sutured, then extend $\gamma$ to $\boundary M'$ naturally, with $(\boundary L)\times I$ contained in elements of $A(\gamma)$.) Let $\F'$ be a taut foliation of $M'$ which agrees with $\F$ outside of $L\times I$, includes $L\times 0$ and $L\times 1$ as leaves, and so every leaf of $\F'|_{L\times (0,1)}$ is transverse to $x\times I$ for each $x\in L$.

Identify $M'$ with $M$. We say the taut foliation $\F'$ on $M$ is obtained from $\F$ by {\emph{I-bundle replacement}} on the leaf $L$. This generalizes the leaf thickening operation.
\end{operation}

\begin{operation}[Suspension change, {\cite[Operation 2.2]{suspensions}}]\label{suspensionchange}
See Figure~\ref{fig:changesuspension} for an illustration of this operation. Let $D$ be a product disk in sutured manifold $(M,\gamma)$. Let $\F$ be a taut foliation on $(M\gamma)$; take $D$ to be transverse to the leaves of $\F$. Assume $\boundary D$ meets two distinct elements $A_0$ and $A_1$ of $A(\gamma)$. Say $\F|_{A_i}$ is a suspension of $f_i:I\to I$ for each $i=0,1$. Choose coordinates $D=I\times I$, so $I\times 0\subset A_0$, $I\times 1\subset A_1$, $S^0\times I\subset R(\gamma)$, and each $x\times I$ is contained in one leaf of $\F$.

Let $V=D\times I$, where the $I$ direction is small and chosen so that $(\boundary D)\times I\subset\boundary M$. Parametrize $V=D\times I=(I\times I)\times I$ so that each $(x\times I)\times I$ is contained in one leaf of $\F$.

Choose a boundary-preserving automorphism $g:I\to I$. We obtain a taut foliation $\F'$ on $M$ by excising $V\cap\int{M}$ (with the foliation $\F|_V$) and regluing $V=I\times I\times I$ via \[\begin{cases}V\ni((x,t),s)\sim ((x,t),s)\in \overline{M\setminus V}&\parbox{4.5cm}{if $x,s\in I$, $t<1$,\\ $((x,t),s)\in M\cap V$}\\[1em]
V\ni((x,1),s)\sim ((f(x),1),s)\in \overline{M\setminus V}&\text{for all $x,s\in I$}.\end{cases}\] 

In words, we cut out $V=D\times I$ and reglue by the identity, except that we shift $D\times 1$ vertically according to the map $g$. This changes the intersection $\F'|_{A_i}$. Now $\F'|_{A_0}$ is a suspension of $f_0\circ g$ while $\F'|_{A_1}$ is a suspension of  $g^{-1}\circ f_1$. In particular, note that if we choose $g=f_0^{-1}$, then $\F'|_{A_0}$ is a suspension of the identity, i.e. a product foliation of compact circles.

\end{operation}

\begin{figure}
\labellist
\small\hair 2pt
\pinlabel $R_+(\gamma)$ at 175 215
\pinlabel $R_-(\gamma)$ at 110 -15
\pinlabel $\F|_V$ at -25 80
\pinlabel $R_+(\gamma)$ at 500 215
\pinlabel $R_-(\gamma)$ at 440 -15
\pinlabel $\F'|_V$ at 640 120
\endlabellist
\begin{centering}
\includegraphics[width=80mm]{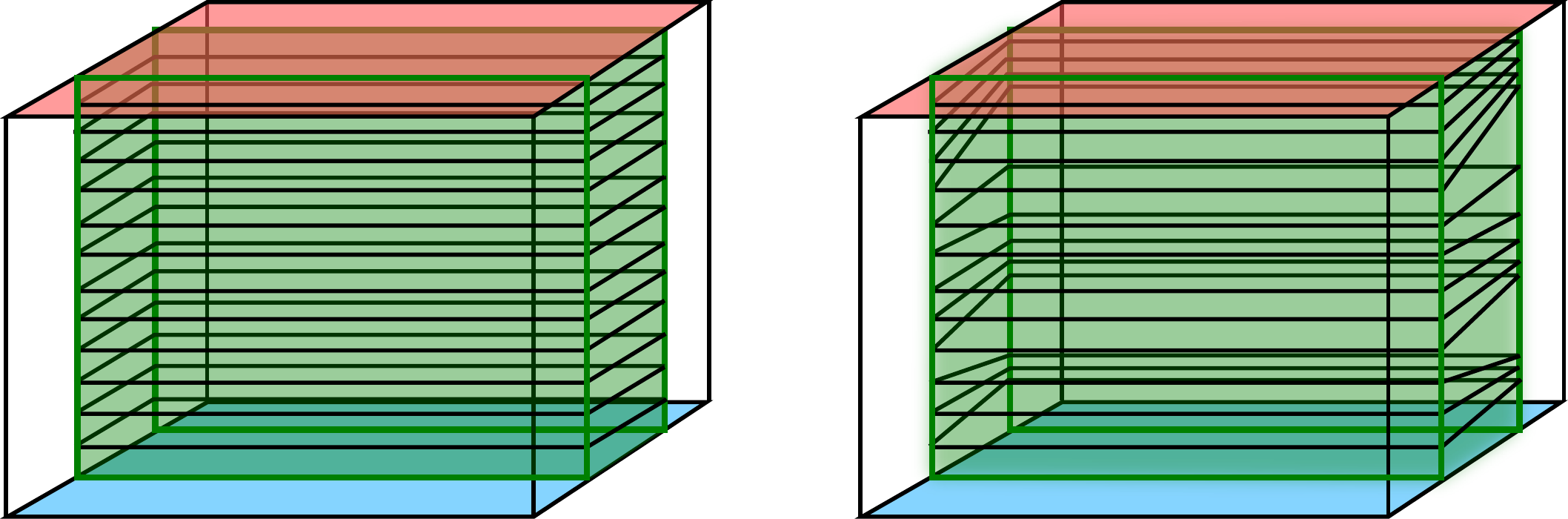}
\caption{{\bf{Left:}} $V=D\times I$, where $D$ is a product disk in a sutured manifold $(M,\gamma)$. We draw the intersection of $V$ with a taut foliation $\F$. {\bf{Right:}} We perform a suspension change operation (Operation~\ref{suspensionchange}) on $V$ to obtain a new taut foliation $\F'$. We draw the intersection of $\F'$ with the reglued $V$. Note that $\F'$ and $\F$ do not agree on the elements of $A(\gamma)$ met by $D$.}\label{fig:changesuspension}
\end{centering}
\end{figure}
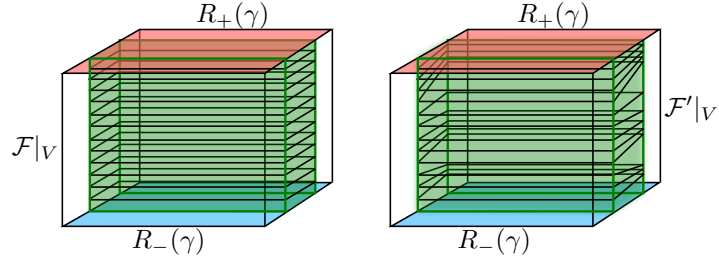

The suspension change operation allows us essential freedom when performing the $I$-bundle replacement operation.

\begin{lemma}[Gabai~\cite{suspensions}]\label{anyhomeo}
Let $F$ be a connected, compact  positive-genus surface with non-empty boundary. Fix some boundary component $C$ of $F$.  Then for any homeomorphism $f:I\to I$, there exists a foliation $\F$ of $F\times I$ transverse to the vertical $I$ fibers so that $\F|_{C\times I}$ is a suspension of $f$ and $\F|_{(\boundary F\setminus C)\times I}$ is a product foliation.

\end{lemma}

\begin{proof}
Let $\alpha$ be a non-separating simple closed curve on $F$. Let $H=F\setminus\nu(\alpha)$. Let $\G$ be the product foliation on $H\times I$.
Since $\alpha$ is non-separating, there exist paths $\beta_1,\beta_2$ in $H$ from $C$ to the distinct components $C_1,C_2$ of $\boundary H\setminus\boundary F$. Then $\beta_1\times I,\beta_2\times I$ are product disks for $\G$ in $H\times I$. See Figure~\ref{fig:genustrick}.

Let $g,h:I\to I$ be automorphisms. By performing the suspension change operation on $\G$ at $\beta_i\times I$, we may find a foliation $\G'$ on $H$ (transverse to the vertical fibers) so that $\G'|_{C\times I}$ is a suspension of $gh$, $\G'|_{C_1\times I}$ is a suspension of $\bar{g}$, and $\G'_{C_2\times I}$ is a suspesnion of $\bar{h}$.

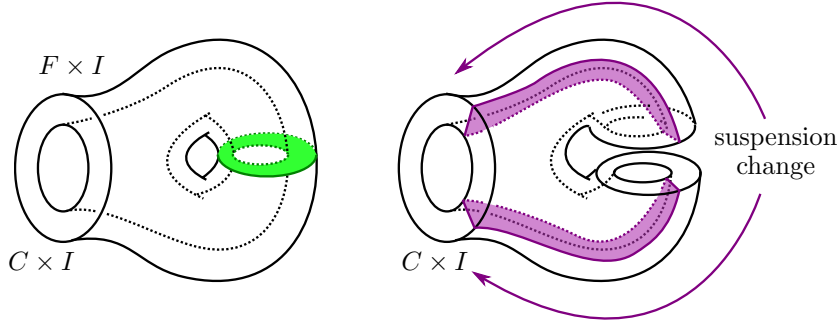
\begin{figure}
\begin{centering}
\labellist
\small\hair 2pt
\pinlabel $F\times I$ at 40 175
\pinlabel $C\times I$ at 19 40
\pinlabel $C\times I$ at 290 40
\pinlabel {suspension} at 525 125
\pinlabel {change} at 525 105
\endlabellist
{\includegraphics[width=100mm]{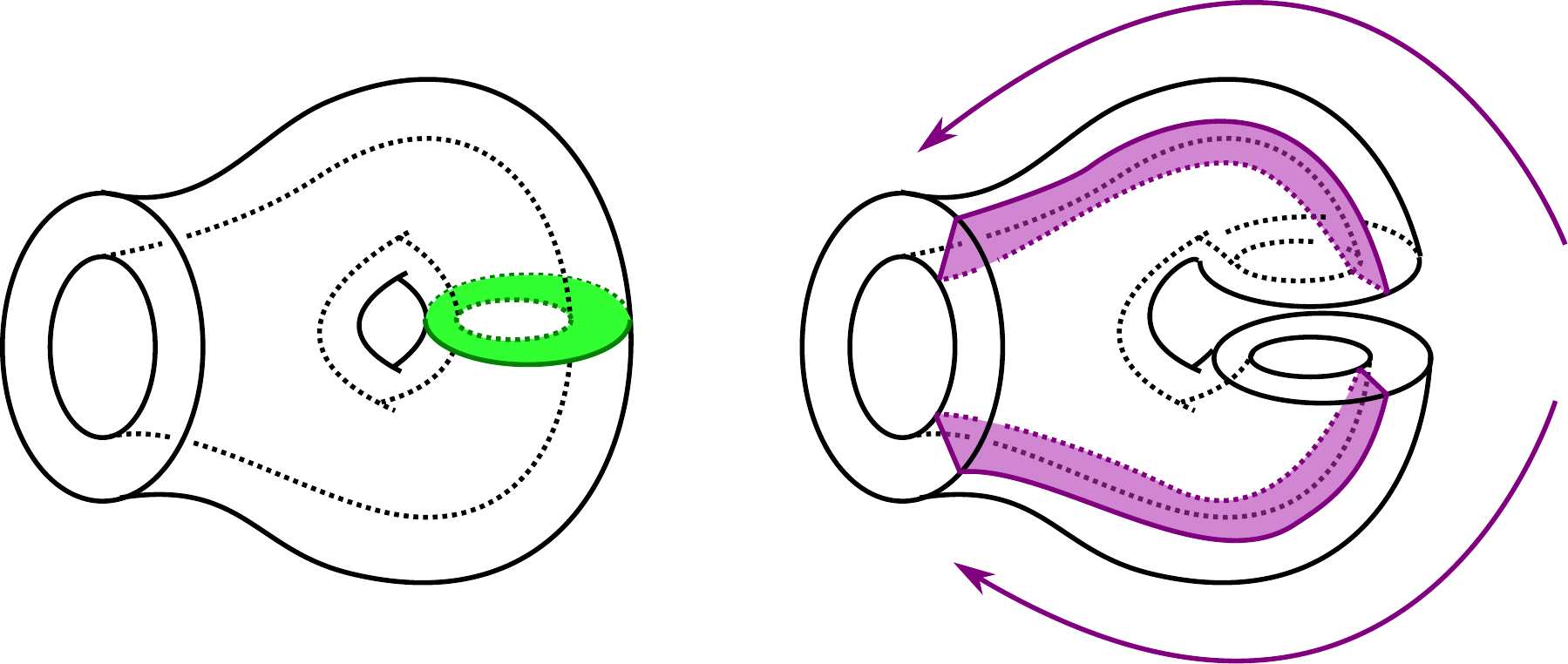}}
\caption{{\bf{Left:}} $F\times I$, where $F$ is a connected positive-genus surface. Let $C$ be a boundary component of $F$. The curve $\alpha$ in $F$ is non-separating. We highlight $\alpha\times I$. {\bf{Right:}} In $H\times I$ (where $H=F\setminus\nu(\alpha)$), we can find product disks connecting $C\times I$ to the two new vertical (transverse to the product foliation $\G$) annuli $C_1\times I,C_2\times I$ in $(\boundary H)\times I$. By performing the suspension change operation on $\G$ at these product disks, we may find a new taut foliation $\G'$ on $H\times I$ so that $\G'|_{C\times I}$ is a suspension of any desired $f:I\to I$. By performing the two suspension changes carefully, as in Lemma~\ref{anyhomeo}, we may arrange that $\G'|_{C_1\times I}$ and $\G'|_{-C_2\times I}$ are suspensions of conjugate automorphisms, so that we may extend $\G'$ to a taut foliation $\F$ of $F\times I$. }\label{fig:genustrick}
\end{centering}
\end{figure}

If $\overline{g}$ and $\overline{\overline{h}}=h$ are conjugate, then we may extend $\G'$ to a taut foliation $\F$ on $F\times I$ by attaching a foliated $(S^1\times I)\times I$ to $H\times I$.

To pick the appropriate $g,h$, assume that $f$ has no fixed points (if $f$ is the identity then the product foliation satisfies the lemma; if $f$ fixes finitely many points then we may repeat this argument on each subinterval in which $f$ has no fixed points). Then $f$ is everywhere expanding or contracting. Let $g$ be a (boundary-preserving) automorphism of $I$ that is also everywhere expanding or contracting, but with the opposite behavior of $f$ (that is, if $f$ expands then $g$ contracts, and vice versa). Take $h=\bar{g}f$. Then $gh=f$, and $g$ and $\bar{h}$ are both every contracting or expanding so are conjugate. The foliation $\F$ is the desired foliation.

\end{proof}

\section{Fat-vertex graphs}\label{sec:graphs}

In this section, we will use a fat-vertex graph to record product disks and find product annuli for a complementary sutured manifold.

\begin{definition}
A {\emph{fat-vertex graph}} is a graph $(V,E)$ on vertices in $V=\{v_1,\ldots, v_m\}$ connected by edges in $E$ along with injective maps $\phi_i:E_{i}\to S^1$ for each $v_i\in V$, where $E_i$ is the set of ends of edges in $E$ at $v_i$. In words, a fat-vertex graph is a graph with a cyclic ordering of edges at each vertex. This cyclic ordering allows us to describe the {\emph{neighborhood}} $\nu(G)$ of $G$. This $\nu(G)$ is a surface with boundary constructed by starting with $n$ disjoint disks $D_1,\ldots, D_m$ in correspondence with $v_1,\ldots, v_m$, and then attaching a band between $D_i$ and $D_j$ for each edge between $v_i$ and $v_j$. The attaching regions of bands along $\boundary D_i$ respects the cyclic ordering of the edges incident to $v_i$.

Note that although we typically use $\nu$ to mean an open neighborhood, in this context it is clear that we are only interested in compact surfaces, so we take $\nu(G)$ to be compact with boundary. We may refer to $D_i$ as $\nu(v_i)$. We refer to each boundary component of $\nu(G)$ as a {\emph{face}} of $G$.
\end{definition}

\begin{construction}[A fat-vertex graph describing product disks and annuli in the complementary sutured manifold to a cut-and-paste surface]\label{construct}

Let $L=L_1\sqcup L_2$ be a $2$-component link in a rational homology sphere $Y$ with nonzero linking number. Let $X=Y\setminus\nu(L)$. The $3$-manifold $X$ has two torus boundary components $P_1$ and $P_2$, where $P_i=\boundary\overline{\nu(L_i)}$.

Let $S$ be a norm-minimizing surface in $X$. 
Let $(M,\gamma)$ be the complementary sutured manifold to $S$ in $X$.

Suppose $\{D_1,\ldots, D_n\}$ is some collection of pairwise disjoint product disks in $(M,\gamma)$, where each disk connects an element of $A(\gamma)$ in $P_1$ to an element of $A(\gamma)$ in $P_2$. Say the sutures of $\gamma$ are $s_1,\ldots, s_m$. Construct a fat-vertex graph $G$ as follows:

Let $G$ have vertices $v_{1},\ldots, v_{m}$ and edges $e_{1},\ldots, e_{n}$. If the product disk $D_i$ connects sutures $s_j$ and $s_k$, then the edge $e_{i}$ goes between vertices $v_{j}$ and $v_{k}$. The cyclic order of the edges at each vertex corresponds to the cyclic order of product disks at each suture. That is, if when travelling (positively) around suture $s_i$ we meet product disks $D_{i_1},\ldots, D_{i_j}$ in order, then when travelling around $v_{i}$ we should find edges $e_{{i_1}},\ldots,e_{{i_j}}$ in order.

Say $G$ has faces $C_1,\ldots, C_f$. We construct product annuli $A_1,\ldots, A_f$ in $(M,\gamma)$ as follows:

\begin{itemize}
 \item Given an arc $a$ in $C_i\cap(\boundary\overline{\nu (v_j)})$ lying close to $v_j$ between edges $e_k$ and $e_l$, let $\Delta_a$ be the corresponding part of the suture $s_j$ lying between product disks $D_k$ and $D_l$. Push $\Delta_a$ slightly into the interior of $M$, so $\boundary \Delta\cap R_+$ is one arc, $\boundary \Delta\cap R_-$ is one arc, and two arcs of $\boundary\Delta$ lie in the interior of $M$.
\item Given an arc $b$ in $C_i$ that is parallel to edge $e_j$, let $\Delta_b$ be a copy of $D_i$ pushed slightly to the side, in the same direction one would push $e_j$ to obtain $b$. 
\item Then if $C_i=\cup_{j=1}^k \alpha_j$, where each $\alpha_j$ is an arc as in one of the two previous bullet points, we form a product annulus $A_i=\cup_{j=1}^k \Delta_{\alpha_j}$. In the construction of each $\Delta_{\alpha_j}$, we chose the pushoffs from $\boundary M$ or $D_l$ so that the edges of the $\Delta_{\alpha_j}$ contained in the interior of $M$ match up, and $A_i$ is an annulus with one boundary component in $R_+$ and the other in $R_-$. See Figure~\ref{fig:productannulus}.
\end{itemize}

\begin{figure}
\labellist
\small\hair 2pt
\pinlabel $G$ at 0 300
\pinlabel $\textcolor{red}{C_i}$ at 100 260
\pinlabel $(M,\gamma)$ at 200 300
\pinlabel $\textcolor{red}{A_i}$ at 320 275
\pinlabel $A(\gamma)$ at 300 20
\endlabellist
\begin{centering}
\includegraphics[width=75mm]{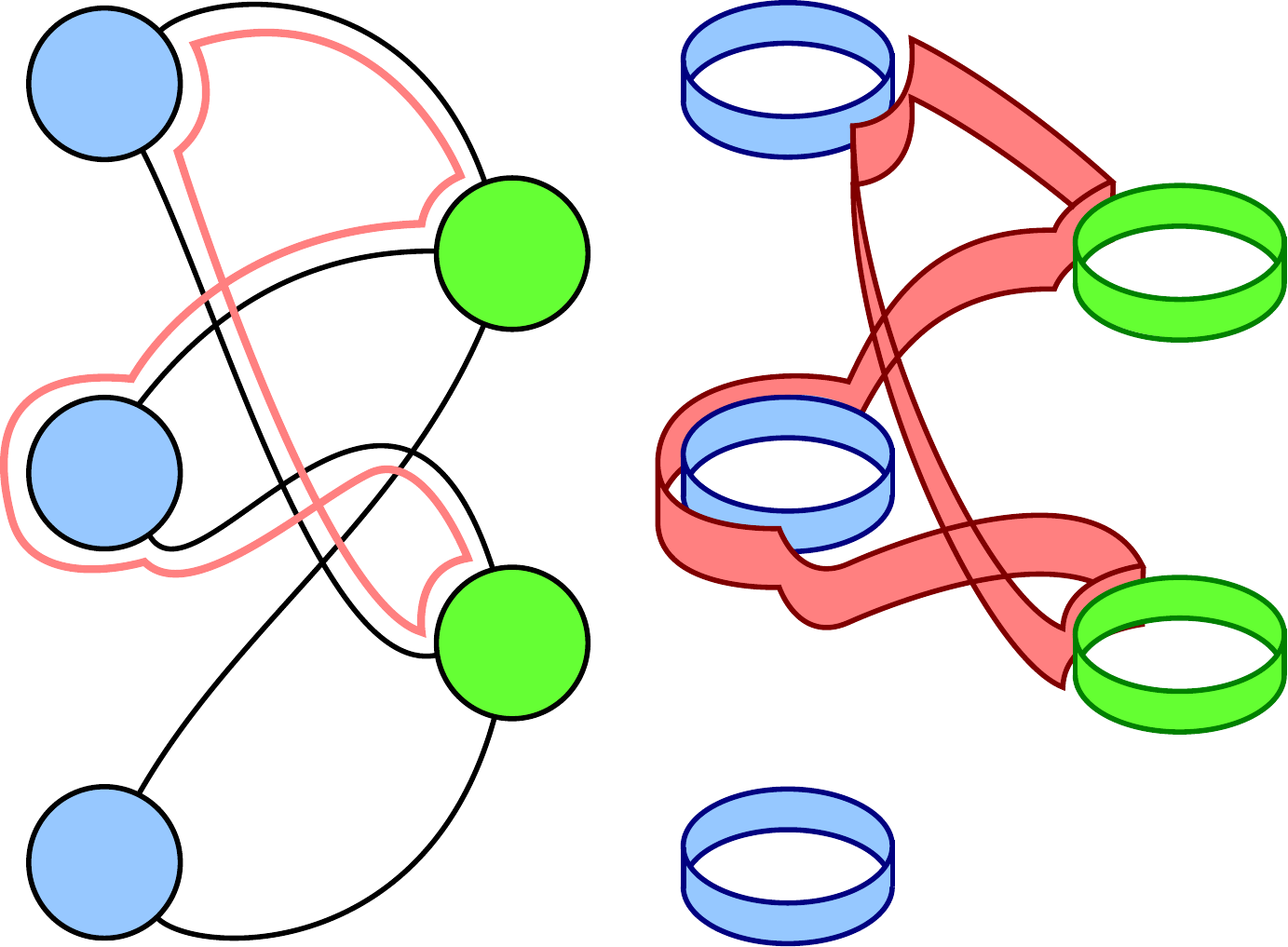}
\caption{Left: a face $C_i$ of $G$, where $G$ is a fat-vertex graph describing product disks in $(M,\gamma)$. Right: we use $C_i$ to construct a product annulus $A_i$ in $(M,\gamma)$.}\label{fig:productannulus}\end{centering}
\end{figure}
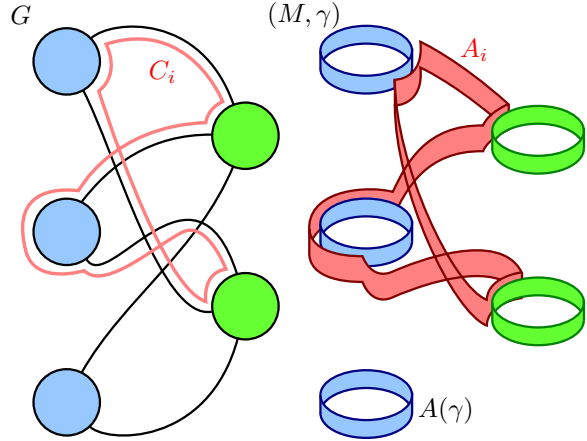

We will continue to use this notation. From now on, $G$ will always refer to a fat-vertex graph describing product disks in a norm-minimizing surface complement in $S^3\setminus\nu(L)$. Later, we will specify the product disks to come from cut-and-paste resolutions as in Proposition~\ref{cutpasteprop}. When we say that a product annulus $A_i$ corresponds to a face $C_i$ of $G$, we will always mean as in this construction.

\end{construction}

\begin{remark}\label{remarkedisks}
Say $G$ has components $G_1,\ldots, G_c$. Let $(M',\gamma')$ be the sutured manifold obtained by decomposing $(M,\gamma)$ by product annulus $A_i$ corresponding to a face $C_i$ of $G$. Say $C_i$ is face of component $G_j$. Let $s_1,\ldots, s_m$ be the sutures of $\gamma$ corresponding to vertices of $G_j$. Then there are disjoint product disks $E_1,\ldots, E_m\subset (M',\gamma')$ so that $E_k$ connects $s_k$ to a suture in $\gamma'\setminus\gamma$ as follows:
\begin{itemize}
 \item Let $q$ be a point in $C_i$. For each $k=1,\ldots, m$, fix a path $p_k$ in $\boundary\nu(G_j)$ connecting a point  in $\boundary\nu(v_k)$ to $q$. 
\item As in the construction of $A_1,\ldots, A_k$ in Construction~\ref{construct}, each $p_k$ gives a product disk meeting the suture $s_k$ and a suture in $\gamma'\setminus\gamma$. Push the product disks $E_1,\ldots, E_m$ slightly off of each other to be disjoint.
\end{itemize}

We draw the product disks $E_1,\ldots, E_m$ in Figure~\ref{fig:edisks}. Note that each $E_k$ meets {\emph{the same}} suture in $\gamma'\setminus\gamma$.

\end{remark}

\begin{figure}
\begin{centering}
\labellist
\small\hair 2pt
\pinlabel $G$ at 0 300
\pinlabel $\textcolor{red}{C_i}$ at 100 265
\pinlabel $q$ at 100 230
\pinlabel $(M,\gamma)$ at 200 300
\pinlabel $\textcolor{red}{A_i\times I}$ at 225 195
\endlabellist
\includegraphics[width=75mm]{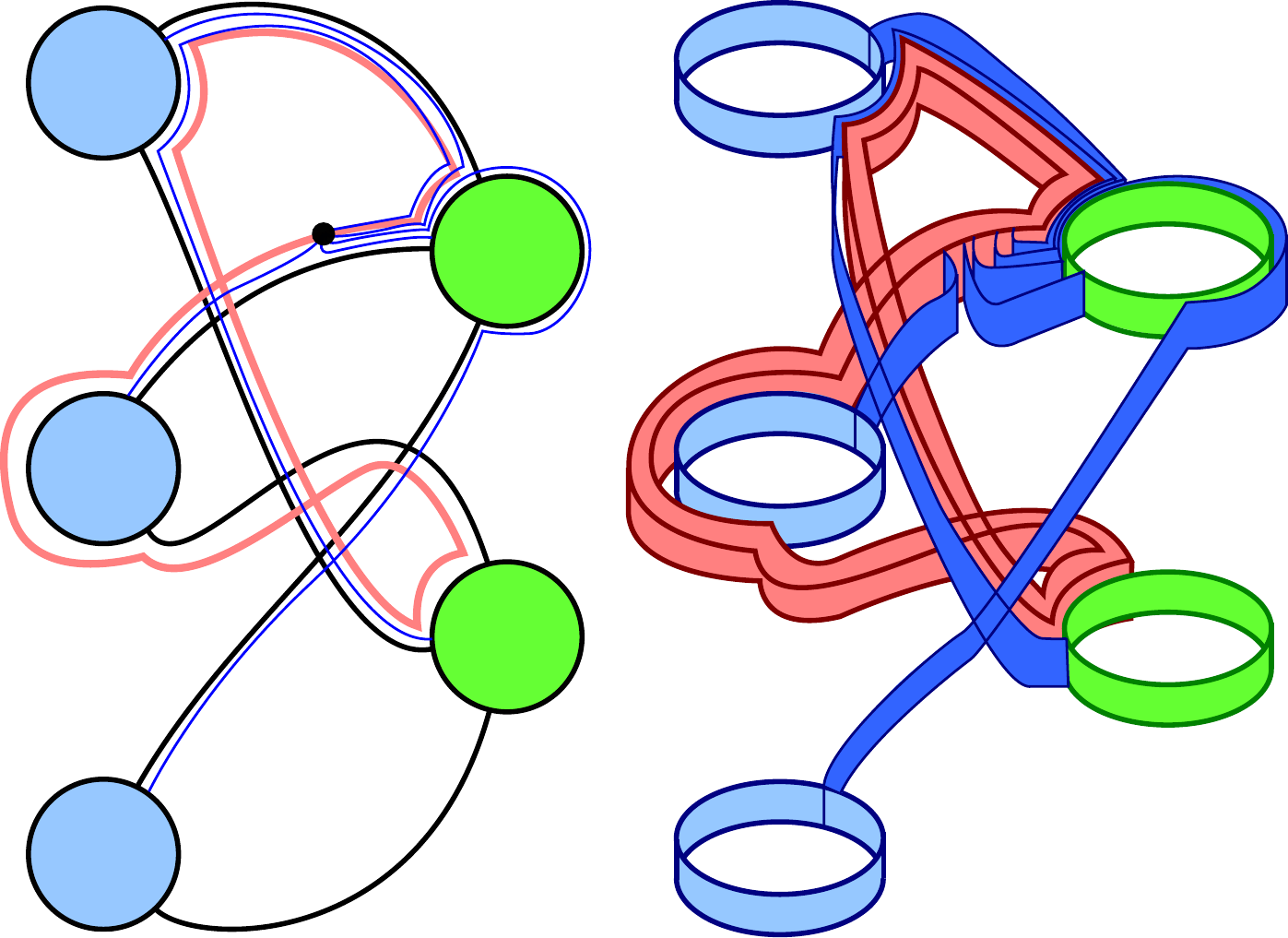}
\caption{Left: paths in $\nu(G)$ connecting each $\boundary\nu(v_k)$ to a point $q$ in $\boundary\nu(G)$. The face containing $q$ corresponds to a product annulus $A$ in $(M,\gamma)$. {\bf{Right:}} The paths yields product disks connecting corresponding elements of $A(\gamma)$ to an element of $A(\gamma')\setminus A(\gamma)$, where $(M,\gamma)\xrightarrow{A_i}(M',\gamma')$.}\label{fig:edisks}\end{centering}
\end{figure}
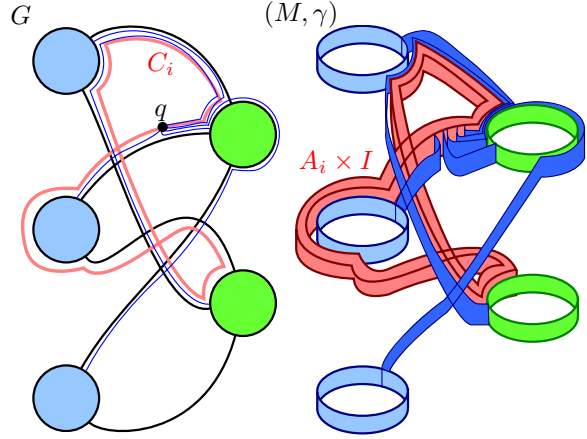

\section{Proof of Theorem~\ref{mingenuscor}}\label{sec:proof}

\begin{mingenuscor}
Let $L=L_1\sqcup L_2$ be an $2$-component link  in a rational homology $3$-sphere $Y$. Assume $\lk(L_1,L_2)\neq 0$ and that $X:=Y\setminus\nu(L)$ has nondegenerate Thurston norm. Let $S$ be a properly norm-minimizing surface in $X$ with $[S]$ primitive and not a corner of the Thurston norm. Let $\widehat{X}$ denote the closed manifold obtained from $X$ by Dehn filling each boundary component of $X$ according to the slope of $\boundary S\cap\boundary X$, and let $\widehat{S}\subset \widehat{X}$ be the closed surface obtained from $S$ by capping off each component of $\boundary S$ by a disk in a Dehn-filling solid torus. 

If $\widehat{S}$ is {\emph{not}} norm-minimizing, then either $g(S)=1$ or $g([S])$ is minimal among all classes in the interior of the same face of the Thurston norm as $[S]$. In particular, if $[S']$ is a primitive class in the interior of the same face as $[S]$ and $\widehat{S'}$ is also not norm-minimizing, then $g([S])=g([S'])$.
\end{mingenuscor}

\begin{proof}

Take $[S]$ to be primitive in $H_2(X,\boundary X;\R)$ and not a corner of the Thurston norm. Then $c[S]=a[R]+b[T]$ for some positive coprime integers $a$ and $b$, an integer $c>0$, and adjacent corners $[R]$ and $[T]$.

\begin{definition}
Let $\alpha_1,\alpha_2$ be the homology classes of punctured Seifert surfaces for multiples of $L_1$ and $L_2$, respectively. Let $\beta_1,\beta_2\in H_2(X;\boundary X)$ so $\beta_1=p\alpha_1+q\alpha_2$ and $\beta_2=r\alpha_1+s\alpha_2$ for some integers $p$, $q$, $r$, and $s$. We say the {\emph{determinant of $\beta_1$ and $\beta_2$}} is \[\det\left(\beta_1,\beta_2\right)=\bigg|\det\left(\begin{matrix}p&q\\r&s\end{matrix}\right)\bigg|.\]	

If $\beta_1$ and $\beta_2$ are vertices of the Thurston norm unit-ball, we call $\det(\beta_1\beta_2)$ the \emph{determinant of the face spanned by $\beta_1$ and $\beta_2$}.

\end{definition}

\begin{proposition}\label{primitive}
If $\det([S_1],[S_2])=1$, then $c=1$.
\end{proposition}
\begin{proof}
Consider the parallelogram $A$ in $\R^2=H_2(X,\boundary X;\R)$ spanned by $[S_1]$ and $[S_2]$. By Pick's Theorem, $\det([S_1],[S_2])=1+($the number of integer lattice points in the interior of $A$). (Here we use the fact that $[S_1]$ and $[S_2]$ are primitive to know the only lattice points on the boundary of $A$ are at the four corners of $A$.)

Thus, if $\det([S_1],[S_2])=1$ then there are no integer lattice points in the interior of $A$, so every lattice point (and hence every element of $H_2(X,\boundary X;\R)$ can be expressed as an integral combination of $[S_1]$ and $[S_2]$,

\end{proof}

Take $R$ and $T$ to be properly norm-minimizing surfaces positioned so that any cut-and-paste of $xR+yT$ is also properly norm-minimizing (for $x,y>0$) as in Proposition~\ref{cutpasteprop}. Let $S^{xy}$ denote the surface $xR+yT$.

Since $\chi(S)=\chi(S')$ and $\chi(\widehat{S})=\chi(\widehat{S'})$ for any properly norm-minizing surface $S'$ with $[S']=[S]$, to conclude that $\widehat{S}$ is norm-minimizing it would be sufficient to show that $\widehat{S'}$ is norm-minimizing. Therefore, we may take $S=S^{ab}$ to be the cut-and-paste sum $aR+bT$.

In the following sections, we will study for which $a,b$ the surface $\widehat{S}^{ab}$ may fail to be norm-minimizing. For notation, let $(M^{ab},\gamma^{ab})$ be the complementary sutured manifold to $S^{ab}$ in $X$. 

Form a fat-vertex graph $G^{ab}$ from this intersection, as in Construction~\ref{construct}. The product disks used to construct the edges of $G^{ab}$ are the disks that arise from the intersections of $R$ and $T$, as in Figure~\ref{fig:cutandpaste}. Since $[S^{ab}]=c[S]$ and $[S]$ is primitive, $G^{ab}$ has $c$ components.

\subsection{Constructing a taut foliation on $Y_{\boundary S^{ab}}(L)$ when possible}\label{sec:constructing}
\subsubsection{Assuming $c=1$}
For now, we will assume $S=S^{ab}$, so $G^{ab}$ is a connected graph. The case that $G^{ab}$ is disconnected is not much harder, but this will simplify notation. We address the general case after understanding connected graphs.

If $X=Y\setminus\nu(L)$ is reducible, then $X$ factors as $(Y'\setminus\nu(L))\# Y''$ for rational homology spheres $Y'$ and $Y''$, with $Y'\setminus\nu(L)$ irreducible. It is sufficient to prove the Theorem for $L\subset Y'$, so we may assume $X$ is irreducible. Therefore, we may apply Theorem~\ref{havefoliation1} to obtain a taut foliation $\F$ on $X$ realizing $S=S^{ab}$ as a leaf.  Restrict $\F$ to a taut foliation on the complementary sutured manifold $(M,\gamma)$ to $S$.

Say $G:=G^{ab}$ has faces  $C_1,\ldots, C_f$. Let $A_1,\ldots, A_f$ be the product annuli in $(M,\gamma)$ 
corresponding to $C_1,\ldots, C_f$, as in Construction~\ref{construct}. Let $S_{\pm}:=R_{\pm}(\gamma)$ (so if $M=\overline{X\setminus(S\times I)}$, $S_+=S\times1$ and $S_-=S\times-1$). We write $\widehat{S},\widehat{S}_+,\widehat{S}_-$ to denote the closed surfaces embedded in $\widehat{X}$ obtained from $S,S_+,S_-$ (respectively) by attaching disks (within the Dehn-filling solid tori) to each boundary component. (We have alredy defined $\widehat{S}$ this way, but we want to be clear that $\widehat{S}_+$ and $\widehat{S}_-$ are defined the same way.)

Let $(M',\gamma')$ be the result of decomposing $(M,\gamma)$ along the product annulus $A_i$. Let $V_i=A_i\times I$ be the solid torus excised in this decomposition. We name the two elements of $A(\gamma)\setminus A(\gamma')$ $B_1$ and $B_2$. Let $\F':=\F|_{M'}$. By Lemma~\ref{tautdecomp} and Remark~\ref{tautdecompfol}, $\F'$ is taut.

Recall from Remark~\ref{remarkedisks} that there are product disks in $(M',\gamma')$ connecting each element of $A(\gamma)$ to one of $B_1$ or $B_2$. 
Perform the suspension change operation (Operation~\ref{suspensionchange}) on these disks to find from $\F'$ a new taut foliation $\G$ on $(M',\gamma')$, where $\G$ restricts to a product foliation on each component of $A(\gamma')\cap A(\gamma)$. (In words, we use product disks to ``push'' the complicatedness of $\F'
$ at each element of $A(\gamma)$ onto $B_i$ instead to obtain $\G$.)

Let $(\widehat{M},\widehat{\gamma})$ be the complementary sutured manifold to $\widehat{S}$ in $\widehat{X}$.
Let $(\widehat{M'},\widehat{\gamma'})$ be the result of decomposing $(\widehat{M},\widehat{\gamma})$ along $A_i$. (The notation here is not misleading -- $(\widehat{M'},\widehat{\gamma'})$ is obtained from $(M',\gamma')$ by attaching $3$-dimensional $2$-handles to each element of $A(\gamma)\subset A(\gamma')$.)

Extend $\G$ to a taut foliation of the sutured manifold $(\widehat{M'},\widehat{\gamma'})$ by capping off leaves of $\G$ with disks.

Say $\G|_{B_1}$ is a suspension of a homeomorphism $f:I\to I$ and $\G|_{B_2}$ is a suspension of a homeomorphism $g:I\to I$.

\begin{goal}
Our current goal is to reglue $V_i$ to $\widehat{M'}$, foliated so as to extend $\G$ to a taut foliation on $\widehat{X}$ achieving $\widehat{S}$ as a leaf. 

We do not expect to always be able to fill $V_i$ and extend $\G$. A priori, we can fill $V_i$ and extend $\G$ when $f$ and $\bar{g}$ are conjugate.
\end{goal}

We will do an $I$-bundle replacement on the leaves $R_+(\widehat{\gamma'}), R_-(\widehat{\gamma'})$ of $\G$ to obtain a taut foliation $\G'$. 
Now the foliation induced by $\G'$ on the intersection of the $I$-bundle $R_\pm(\widehat{\gamma'})\times I$ with $B_j$ a suspension of some $\mu_j^\pm:I\to I$. See Figure~\ref{fig:suspension}. We have not yet specified our choice of foliation on the $I$-bundles to determine each $\mu_j^\pm$. 

We will consider some situations that allow us to choose some of the $\mu_*^*$ freely. Note $\G'|_{B_1}$ is a suspension of the concatenation $\mu_1^-f\mu_1^+$ and $\G'|_{B_2}$ is a suspension of the concatenation $\mu_2^-g\mu_2^+$. If we can choose $\mu_*^*$ so that $\mu_1^-f\mu_1^+$ and $\bar{\mu}_2^-\bar{g}\bar{\mu}_2^+$ are conjugate, then we can fill $V_i$ and extend $\G'$, proving that $\widehat{S}$ is norm-minimizing. We will use Lemma~\ref{anyhomeo} to choose the suspension homeomorphisms $\mu_{\pm}^j$ appropriately.

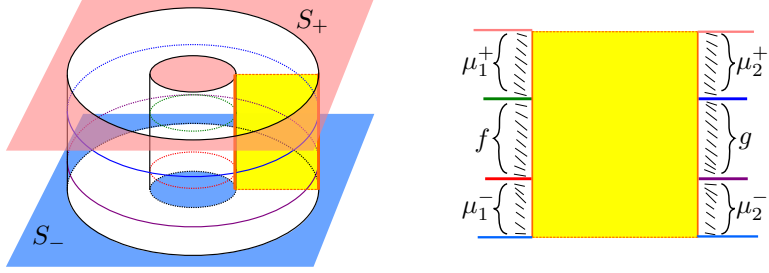
\begin{figure}
\begin{centering}
\labellist
\small\hair 2pt
\pinlabel $S_+$ at 425 340
\pinlabel $S_-$ at 60 40
\pinlabel $\mu_1^+$ at 655 285
\pinlabel $f$ at 660 175
\pinlabel $\mu_1^-$ at 655 85
\pinlabel $\mu_2^+$ at 1035 285
\pinlabel $g$ at 1025 175
\pinlabel $\mu_2^-$ at 1035 85
\endlabellist
{\includegraphics[width=100mm]{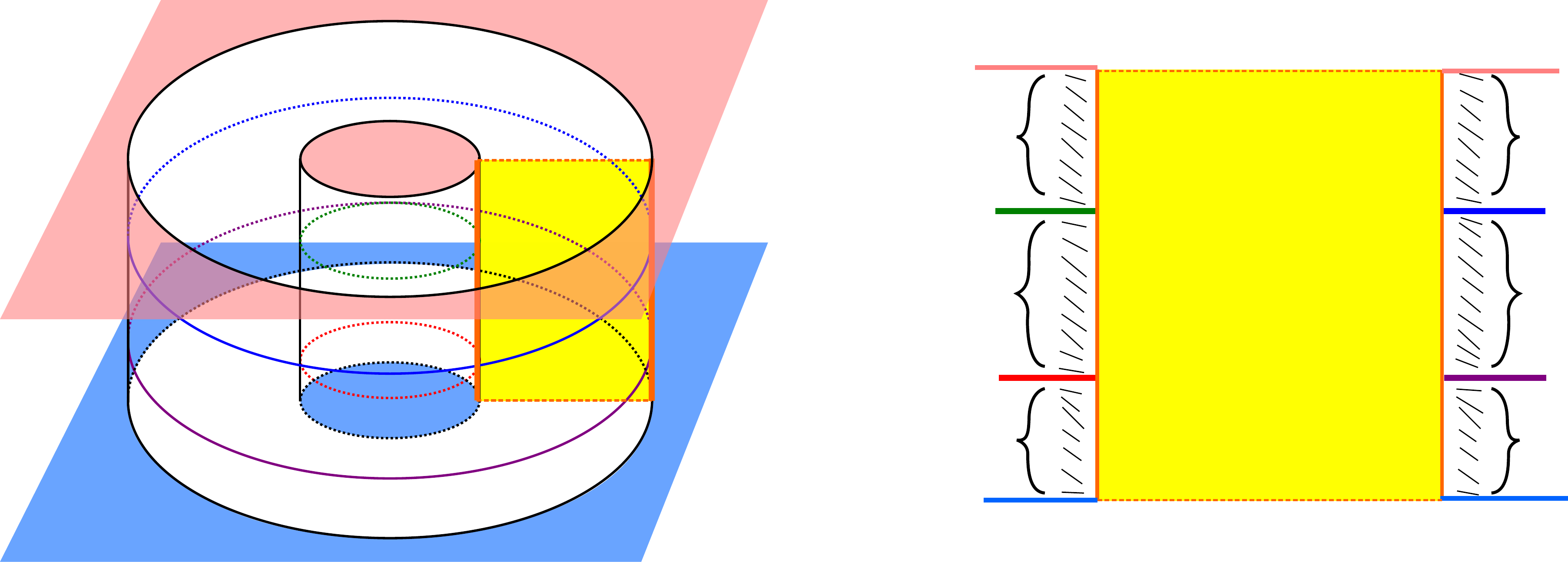}}\caption{{\bf{Left:}} the boundary of $V_i$ in $(\widehat{M'},\widehat{\gamma'})$. We highlight where we would like to glue a disk when filling $\boundary V_i$. {\bf{Right:}} The foliation $\G$ on the two annuli in $A(\gamma')$. Each annulus is subdivided into three annuli on which $\G$ induces some suspension foliation. We can fill $\boundary V_i$ by a solid torus and extend $\G$ if the concatenations $\mu_1^-f\mu_1^+$ and $\bar{\mu}_2^-\bar{g}\bar{\mu}_2^+$ are conjugate as automorphisms of $I$.}\label{fig:suspension}\end{centering}
\end{figure}

If $g(S)=1$ whenever $\widehat{S}$ is norm-minimizing, then the theorem holds automatically. So assume $g(S)>1$.

\par {\bf{Case 1. Neither $\boundary_+A$ nor $\boundary_-A$ are separating in $\widehat{S}_+,\widehat{S}_-$.}}

Then $\widehat{S}_{\pm}\setminus\nu(A)$ is positive genus, so by Lemma~\ref{anyhomeo} we may choose e.g. $\mu_+^1=\bar{g},\mu_-^2=\bar{f},\mu_+^2=\mu_-^1=\id$.

Now $\mu_-^1 f \mu_+^1 =\id f \bar{g}$ and $\bar{\mu_-^2}\bar{g}\bar{\mu_+^2}=f\bar{g}\id$ are conjugate (in fact isotopic) and hence $V_i$ can be filled to extend $\G$ to a taut foliation on $(\widehat{M},\widehat{\gamma})$, 
proving that $\widehat{S}$ is norm-minimizing.

\par {\bf{Case 2. $\boundary_+A $ is separating in $\widehat{S}_+$, but $\boundary_-A $ is non-separating in $\widehat{S}_-$.}}

(The case in which $\boundary_-A$ is separating but $\boundary^+A$ is non-separating is similar.)

Let $S_{+}'$ be the component of $S_+\setminus V$ meeting $B_1^+$  and let $S_+^2$ be the component meeting $B_2^+$. Implicitly, we choose the order of the $B_1$ and $B_2$ so that $B_1^+$ contains $\boundary S_+$.

Since $\boundary_-A $ is non-separating in $\widehat{S}_-$ (and hence in $S_-$) and $S_-$ is incompressible in $X$, $\boundary_+A $ does not bound a disk in $S_+$. Therefore, at least one of the $\widehat{S}_+^j$ (say WLOG $\widehat{S}_+^2$) is positive-genus. We then choose $\mu_+^2=\bar{f},\mu_{-}^1=\bar{g},\mu_+^1=\mu_-^2=\id$. 
Now $\mu_-^1 f \mu_+^1 =\bar{g}f\id$ and $\bar{\mu_-^2}\bar{g}\bar{\mu_+^2}=\id \bar{g}f$ are conjugate (in fact isotopic) and hence $V_i$ can be filled to extend $\G$ to a taut foliation on $(\widehat{M},\widehat{\gamma})$, 
proving that $\widehat{S}$ is norm-minimizing.

\par {\bf{Case 3. Both $\boundary_+A$ and $\boundary_-A$ are separating in $\widehat{S}_+,\widehat{S}_-$.}}

Let $S_{\pm}'$ be the component of $S_\pm\setminus V$ meeting $B_1^\pm$  and let $S_\pm^2$ be the component meeting $B_2^+$. Implicitly, we choose the order of the $B_1$ and $B_2$ so that $B_1^\pm$ contains $\boundary S_\pm$.

\par {\bf{Case 3(i). Both $S_+^1$ and $S_-^2$ have positive genus.}}
(The case in which $\widehat{S}_-^1$ and $\widehat{S}_+^2$ have positive genus is similar.)

 Choose $\mu_+^1=\overline{g},\mu_-^2=\overline{f}$, and $\mu_-^1=\mu_+^2=\id$. Now $\mu_-^1 f \mu_+^1 =\id f\bar{g}$ and $\bar{\mu_-^2}\bar{g}\bar{\mu_+^2}=f\bar{g}\id$ are conjugate (in fact isotopic). Then $\G$ extends over $V_i$ as desired, so $\widehat{S}$ is norm-minimizing.

\begin{remark}
The remaining cases describe the situation in which we cannot apply Lemma~\ref{anyhomeo} to control either of $\mu_-^1$ and $\mu_+^1$, or either of $\mu_-^2$ and $\mu_+^2$. When this is the case, we cannot immediately cause the two suspension foliations on $B_1$ and $B_2$ to be conjugate. We may then consider other values of $i$, i.e. consider other product annuli constructed from different faces of the fat-vertex graph $G$. If any annulus yields the situation of Case 1, 2, or 3(i), then $\widehat{S}$ is norm-minimizing. When every annulus fails to yield a situation of a previous case, then we may still consider whether each annulus exhibits similar behavior.
\end{remark}

\par {\bf{Case 3(ii). For every choice of $i$, $S_+^2$ is a disk.}}

Since $S$ is incompressible, when $S_+^2$ is a disk it must be the case that $S_-^2$ is a disk. Note that $A_i\cup($disks in $S_\pm^2)$ is a $2$-sphere. Decomposing $(\widehat{M},\widehat{\gamma})$ along all $A_1,\ldots, A_f$ yields a sutured manifold $(\nu(G)\times I,(\boundary\nu(G))\times I)\sqcup(N,\eta)$, where $(N,\eta)$ is potentially disconnected. Moreover, $\boundary N$ is a collection of $2$-spheres which each contain one element of $A(\eta)$. The sutured manifold $(\nu(G)\times I,(\boundary\nu(G))\times I)$ is taut (since it is a product) and $(N,\eta)$ is taut (since $R_\pm(\eta)$ is norm-minimizing), so we conclude by Lemma~\ref{tautdecomp} that $(\widehat{M},\widehat{\gamma})$ is taut. Thus, $\widehat{S}$ is norm-minimizing.

\begin{remark}
If $Y=S^3$, then in Case 3(ii) we may actually conclude that $\widehat{S}$ is a fiber surface by~\cite[Theorem 3.14]{generaoflinks}.
\end{remark}

Thus, we consider the last remaining possibility.

\par {{\bf{Case 3(iii).}} For every choice of $A_i$, either $S_+^1$ is genus-zero or both $S_\pm^2$ are disks, and we are not in Case 3(ii).} 
In this case, we either cannot control $\mu_\pm^1$ or cannot control $\mu_\pm^2$.

Let $\Sigma_1,\ldots,\Sigma_f$ be the regions of $S_+$ bounded by $\boundary_+ A_i$ which do not meet $\boundary S_+$. These regions are disjoint. 

Some $\Sigma_j$ must be positive-genus, since we are not in Case 3(ii). By assumption, this means that $\cup_{k\neq j}\Sigma_k\subset S_+\setminus\Sigma_j$ is genus-zero. Therefore, $g(S)=g(\Sigma_j)$.

If $\widehat{S}$ is not norm-minimizing, then the product annuli corresponding to faces of $G$ must be as in this Case (3(iii)). As a shorthand, we will just say that ``$S$ is type 3(iii)'' or ``$[S]$ is type 3(iii).''

\subsection{When $G^{ab}$ is disconnected}
Now suppose $c[S]=a[R]+b[T]$, so $S^{ab}$ has $c$ components. Take $S$ to be one of these components. Now the complementary sutured manifold to $\widehat{S}^{ab}$ has $c$ components. Of these, $(c-1)$ are product sutured manifolds (these are the regions between parallel copies of $S$). The other component is a copy of the complementary sutured manifold $(\widehat{M},\widehat{\gamma})$ to $\widehat{S}$. Let $G$ be the component of $G^{ab}$ whose vertices and edges correspond to elements of $A(\widehat{\gamma})$ and product disks in $(\widehat{M},\widehat{\gamma})$. Then repeat the analysis of Subsection~\ref{sec:constructing} using $G$ in the place of $G^{ab}$. Again, we find that $(\widehat{M},\widehat{\gamma})$ is taut as long as it is not the case that for each of the product annuli $A_i$ corresponding to a face of $G$, $\boundary_\pm A_i$ bounds either a disk or a surface of genus-$g(S)$ in $S_\pm$. In this case in which we do not know that $(\widehat{M},\widehat{\gamma})$ is taut, we again say ``S is type 3(iii)'' or ``$[S]$ is type 3(iii)'' as a convenient shorthand.

\subsection{Limiting Case 3(iii)}\label{sec:limiting}

Now we will understand which homology classes in $S$ may be type 3(iii).

Recall that if $[S]\in H_2(X,\boundary X;\R)\setminus E$ is primitive, then we can take a properly norm-minimizing surface $S$ to be $S=S^{ab}:=aR+bT$ for some positive coprime $a,b\in\R$, where $[R],[T]\in E$ are corners of the Thurston norm on $X$.

\begin{definition}
Say a component of $P_i\setminus(S^{ab})$ is a {\emph{corner}} if it does not lie between two parallel copies of $R$ or $T$. Note an annulus in corresponding to a face of $G^{ab}$ is parallel to $P_i$ either only in corners or never in corners. If such an annulus is parallel to $P_i$ only in corners, we call the annulus a {\emph{corner annulus}}. Otherwise, we call the annulus a {\emph{noncorner annulus}}. See Figure~\ref{fig:corners}.

\end{definition}

\begin{figure}
\begin{centering}
\labellist
\small\hair 2pt
\pinlabel $\textcolor{blue}{\boundary T}$ at 45 100
\pinlabel $\textcolor{red}{\boundary R}$ at 82 40
\endlabellist
{\includegraphics[width=110mm]{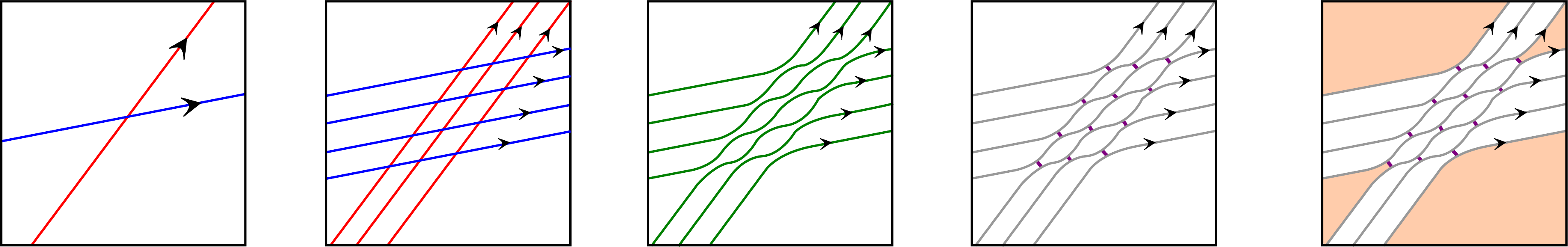}}
\caption{{\bf{First:}} On $P_i$, an intersection of $\boundary R$ and $\boundary T$, where $R$ and $T$ are properly norm-minimizing surfaces. {\bf{Second}}: One intersection of $\boundary R$ and $\boundary T$ gives rise to $ab$ intersections of $\boundary aR$ and $\boundary bT$. {\bf{Third:}} On $P_i$, $\boundary(aR+bT)$. {\bf{Fourth:}} We indicate product disks in the complementary sutured manifold to $aR+bT$ arising from the cut-and-paste construction. Specifically, we draw the intersection of these disks with $P_i$. {\bf{Fifth:}} We shade the visible regions of $P_i$ which are corners. Annuli corresponding to faces of $G^{ab}$ which are somewhere parallel to these regions are called corner annuli.}\label{fig:corners}
\end{centering}
\end{figure}

\begin{proposition}\label{mingenusprop}

If $S=S^{ab}$ is type 3(iii), then $g(S)\le g(S^{xy})$ for all $x,y$.
\end{proposition}

\begin{proof}
Say the annuli corresponding to faces of $G^{ab}$ bound regions $\Sigma_1,\ldots,\Sigma_n$ in the interior of $S_+$. 
Note that $S^{xy}$ contains subregions homeomorphic to each $\Sigma_i$, so $g(\Sigma_i)\le g(S^{xy})$. Since S is type 3(iii), for some $i$ we have $g(\Sigma_i)=g(S)$. Therefore, $g(S)\le g(S^{xy})$.

\end{proof}

Let $c_{xy}$ denote the number of components of $S^{xy}$ for $x,y$ positive integers. (Since $[S]$ is primitive and $[S^{ab}]=c[S]$, $c_{ab}=c$.)

\begin{proposition}
If $[S^{ab}]$ is type 3(iii), then $\frac{1}{c_{ab}}g(S^{ab})\le\frac{1}{c_{xy}}g(S^{xy})$ for all coprime $x,y>0$.
\end{proposition}
\begin{proof}
Say the annuli corresponding to faces of $G^{ab}$ bound regions $\Sigma_1,\ldots,\Sigma_n$ in the interior of $S_+$, where by $S_+$ we mean the component of $S^{ab}_+$ which is not between parallel copies of $S$. Because $x,y>0$, $S^{xy}$ contains subregions homeomorphic to each $\Sigma_i$, so $g(\Sigma_i)\le \frac{1}{c_{xy}}g(S^{xy})$. Since $S^{ab}$ is type 3(iii), for some $i$ we have $g(\Sigma_i)=\frac{1}{c_{ab}}g(S^{ab})$. Therefore, $\frac{1}{c_{ab}}g(S^{ab})\le\frac{1}{c_{xy}}g(S^{xy})$.

\end{proof}

Thus, if $\widehat{S}$ is not norm-minimizing then either $g(S)=1$, or it is the case that $g(S)\le g(\alpha)$ for every $\alpha$ in the face of the Thurston norm which is spanned by $[R]$ and $[T]$. This completes the proof of Theorem~\ref{mingenuscor}.

\end{proof}

\begin{corollary}\label{scholium}
In Theorem~\ref{2compthm}, if $[L_1]=[L_2]=0\in H_1(Y;\Z)$ and $\lk:=\lk(L_1,L_2)=\pm1$ and every face of the unit-norm ball of the Thurston norm $x$ on $Y\setminus\nu(L)$ has determinant $1$, then $\widehat{S}$ is norm-minimizing if $[S]$ is not a vertex of $x$ or a bisector of a face of $x$.	
\end{corollary}
\begin{proof}
Take $[S]$ not a corner of $x$ nor a bisector of a face of $x$. Since $\det(aR,bT)=1$, we may take $S=aR+bT$ for some adjacent corners $[R],[T]$ and positive coprime integers $a,b$ with $(a,b)\neq(1,1)$. Because $|\lk|=1$, $\boundary S$ has two components. This, together with nondegeneracy of $x$, implies $g(S)>g(R+T)$. Moreover, $\chi(S)=a\chi(R)+b\chi(T)\le -3$, so $g(S)>1$. Then by Theorem~\ref{mingenuscor}, $\widehat{S}$ must be norm-minimizing.
\end{proof}

Now we give bounds on $|E|$ in Theorem~\ref{2compthm}. In Corollaries~\ref{cor1} and~\ref{cor2}, the actual upper bounds given on $|E|$ could be improved with some mild arithmetic or by investigating the geometry of a specific Thurston unit-norm ball. Here we just give some easy bounds to make clear that a bound on $|E|$ can be made explicit.

\begin{corollary}\label{cor1}
In Theorem~\ref{2compthm}, if $[L_1]=[L_2]=0\in H_1(Y;\Z)$, then we may take $E$ to contain at most $[n_\mu+\sum_{i=1}^{n_x}(\det(\alpha_i,\alpha_{i+1})-1)](1+{{|\lk|}\choose{2}})$ elements, where $n_x$ is the number of corners of the Thurston norm $x$ on $X$, $\alpha_1,\ldots,\alpha_{n_\mu}$ are the corners of $x$ in cyclic order (taking $\alpha_{n_\mu+1}:=\alpha_1$), and $\lk=\lk(L_1,L_2)$.
\end{corollary}
\begin{proof}
We take $E$ to contain the corners $\alpha_1,\ldots,\alpha_{n_x}$ of $x$. Then we add any integral classes in the parallelogram spanned by $\alpha_i,\alpha_{i+1}$; this consists of exactly $\det(\alpha_i,\alpha_{i+1})-1$ elements. Now if $[S]\not\in E$ then we can take $S=aR+bT$ for some $[R],[T]\in E$. We have $|\boundary S|\le|\lk|+1$, so for $a+b>|\lk|+1$ we have $g(S)>g(R+T),1$. Now let $F=\{a\beta_1+b\beta_2\mid \beta_1,\beta_2$ are cyclically adjacent in $E$ with $a,b\ge1$ and $a+b\le |\lk|+1$. The set $F$ contains $|E|{{|\lk|}\choose{2}}$ elements. Let $E:=E\cup F$.
\end{proof}

\begin{corollary}\label{cor2}
In Theorem~\ref{2compthm}, if $m_1,m_2$ are the smallest positive integers with $m_1[L_1]=m_2[L_2]=0\in H_1(Y;\Z)$, then we may take $E$ to contain at most $[n_x+\sum_{i=1}^{n_x}(\det(\alpha_i,\alpha_{i+1})-1)](1+{{C-1}\choose{2}})$ elements, where $n_x$ is the number of corners of the Thurston norm $x$ on $X$, $\alpha_1,\ldots,\alpha_{n_x}$ are the corners of $x$ in cyclic order (taking $\alpha_{n_x+1}:=\alpha_1$), and $C=2|\lk(L_1,L_2)|m_1^2m_2^2$.
\end{corollary}
\begin{proof}
We take $E$ to contain the corners $\alpha_1,\ldots,\alpha_{n_x}$ of $x$. Then we add any integral classes in the parallelogram spanned by $\alpha_i,\alpha_{i+1}$; this consists of exactly $\sum_{i=1}^{n_x}\det(\alpha_i,\alpha_{i+1})-1$ elements. Recall from~\ref{g0prop} that $|\boundary S|\le C$, so for $a+b>C$ we have $g(S)>g(R+T),1$.  Now let $F=\{a\beta_1+b\beta_2\mid \beta_1,\beta_2$ are cylically adjacent in $E$ with $a,b\ge1$ and  $a+b\le C\}$. The set $F$ contains ${{C-1}\choose{2}}|E|$ elements. Let $E:=E\cup F$.
\end{proof}

\section{Links with many components}\label{sec:ncomp}

Now we consider links of more than two components.

\begin{mingenuscor2}
Let $L=L_1\sqcup L_n$ be an $n$-component link in a rational homology $3$-sphere $Y$. Assume $\lk(L_i,L_j)\neq 0$ for each $i\neq j$. Let $X:=Y\setminus\nu(L)$. Assume $X$ has nondegenerate Thurston norm and is atoroidal and anannular.

Let $S$ be a norm-minimizing surface so that $[S]$ is primitive and is contained in a cone $C$ on the interior of a face of the unit-ball of $x$. 

Let $\widehat{X}$ be the closed manifold obtained by Dehn-filling each boundary component of $X$ according to the slope of $\boundary S$. Let $\widehat{S}$ be the closed surface in $\widehat{X}$ obtained from $S$ by capping off each boundary component of $S$ by a disk in a Dehn-filling solid torus. Then at least one of the following is true:
\begin{itemize}
\item $g(S)=1$,
\item $\widehat{S}$ is norm-minimizing,
\item $g(S)\le g(\beta)$ whenever $\beta\in C$.
\end{itemize}
\end{mingenuscor2}

\begin{proof}
Assume $g(S)>1$ and that $\widehat{S}$ is not norm-minimizing. Let $S'$ be a properly norm-minimizing surface in $X$ with $[S']\in C$. Then there exist adjacent corners $[R],[T]$ of $x$ so that $c[S]=a[R]+b[T],c'[S']=a'[R]+b'[T]$ for some positive integers $a,b,c,a',b',c'$. Take $R$ and $T$ to be properly norm-minimizing, and take $cS=aR+bT, c'S'=a'R+b'T$. 

If $\boundary R$ and $\boundary T$ meet some components $P_i$ of $\boundary X$ in the same slope $q$ or one in slope $q$ and the other not at all, then $S$ and $S'$ meet $P_i$ in slope $q$. Let $\widetilde{X}$ be obtained from $X$ by filling $P_i$ with slope $q$, and let $\widetilde{S},\widetilde{S}',\widetilde{R},\widetilde{T}$ be the corresponding capped surfaces in $\widetilde{X}$. By Theorem~\ref{fillonethm2}, $\widetilde{S}$ is norm-minimizing. Moreover, if $R$ and $T$ meet $P_i$ in $b_R,b_T$ curves correspondingly, then $c\chi(\widetilde{S})=\chi(S)+ab_R+bb_T=a(\chi(R)+b_R)+b(\chi(T)+b_T)=a\chi(\widetilde{R})+b\chi(\widetilde{T})$. Since $\widetilde{S}$ is norm-minimizing, both $\widetilde{R}$ and $\widetilde{T}$ must be norm-minimizing. Similarly, $c'\chi(\widetilde{S})=a'\chi(\widetilde{R})+b'\chi(\widetilde{T})$. Then by convexity of Thurston norm, $\widetilde{S}'$ is norm-minimizing. Now let $X:=\widetilde{X},R:=\widetilde{R},T:=\widetilde{T},S:=\widetilde{S}, S':=\widetilde{S}'$, $n:=n-1$. Continue until both $\boundary R$ and $\boundary T$ meet every component of $\boundary X$, and always do so in distinct slopes.

Let $G$ be the fat-vertex graph describing the intersection of $aR$ and $bT$, as in Construction~\ref{construct}. By the construction of Subsection~\ref{sec:constructing}, $[S]$ must be type 3(iii). That is, there must be a product annulus $A$ (in the complementary sutured manifold to $S$) corresponding to a face of $G$ so that $\boundary_+(A)$ bounds a genus-$g(S)$ region $\Sigma$ of $\int{S}_+$. But since $a',b'>0$, the cut-and-paste $c'S'=a'R+b'T$ must contain a subregion homeomorphic to $\Sigma$. Therefore, $g(S')\ge g(S)$.
\end{proof}

Finally, we are ready to induct on Theorem~\ref{2compthm}.

\begin{ncompthm}
Let $L$ be an $n$-component link $(n>1)$ in a rational homology $3$-sphere $Y$ with components $L_1,\ldots, L_n$. Assume $\lk(L_i,L_j)\neq 0$ for $i\neq j$ and $X:=Y\setminus\nu(L)$ has nondegenerate Thurston norm.

Let $S$ be a properly norm-minimizing surface in $X$ meeting every component of $\boundary X$ . Let $\widehat{X}$ be the closed manifold obtained from $X$ by Dehn filling $X$ according to $\boundary S$, and let $\widehat{S}\subset\widehat{X}$ be the closed surface obtained from capping off each boundary component of $S$ within the Dehn-filling solid tori.

Let $\widetilde{Y}$ be the $3$-manifold obtained from $Y$ by surgering $Y$ along $L_3\sqcup\cdots\sqcup L_n$ according to $\boundary S$.

There exists an $(n-2)$-dimensional set of rays $E$ from the origin of $H_2(X,\boundary X;\R)\cong\R^n$ so that if $[S]\not\in E$, then either $\widehat{S}$ is norm-minimizing or $\widetilde{Y}\setminus\nu(L_1\sqcup L_2)$ has degenerate Thurston norm.
\end{ncompthm}

\begin{proof}[Proof of Theorem~\ref{ncompthm}]


Let $E$ contain the corners of the Thurston norm $\mu$. Since $\mu$ is nondegenerate, the corners of $\mu$ are contained in an $(n-2)$-dimensional set of rays in $H_2(X,\boundary X;\R)$ (e.g. if $n=3$, then the unit ball of $\mu$ is a $3$-dimensional polyhedron. The corners are scalar multiples of vertices and edges, so live in a $1$-dimensional set of rays).

Assume $S\not\in E$, so $c[S]=a[R]+b[T]$ for two adjacent corners $[R],[T]$ and positive integers $a,b,c$. Choose $[R]$ and $[T]$ so that $\boundary R$ and $\boundary T$ meet $P_i$ in distinct (nonmultiple) slopes for each $i=1,\ldots, n$ (this holds outside of some $(n-1)$-dimensional subspaces of $H_2(X,\boundary X;\R)$ consisting of an $(n-2)$-dimensional collection of rays). Take $R$ and $T$ to be properly norm-minimizing surfaces with $S=aR+bT$.

Let $E$ contain all classes $\alpha\in H_2(X,\boundary X;\R)$ with $\boundary\alpha\cap P_n$ of slope zero. Since $\lk(L_n,L_j)\neq 0$ for each $j<n$, this another $(n-2)$-dimensional collection of rays. We add these classes to $E$ to ensure that Dehn-filling $P_n$ according to $\boundary S$ yields a link complement in a rational homology sphere. 

Let $\lk_{ij}:=\lk(L_i,L_j)$. Note that after surgery of slope $q\in\Q$ on $L_n$, for $i\neq j>n$ the link components $L_i,L_j$ become $\widetilde{L}_i,\widetilde{L}_j\subset Y_q(L_n)$ with $\lk(\widetilde{L}_i,\widetilde{L}_j)=\lk_{ij}-q\lk_{in}\lk_{jn}$. For each $i\neq j<n$, let $E$ also contain all $\alpha\in H_2(X,\boundary X;\R)$ with $\boundary\alpha\cap P_n$ of slope $\lk_{ij}/(\lk_{in}\lk_{jn})$. This is similarly another (finite collection of) $(n-2)$-dimensional collection of rays. We add these classes to $E$ to ensure that Dehn-filling $P_n$ according to $\boundary S$ yields a link complement of a link whose linking numbers are pairwise nonzero.

Let $X_q$ be the result of Dehn-filling $P_n$ with slope $q$.

The claim holds for $n=2$ by Theorem~\ref{2compthm}. (Note either Theorem~\ref{2compthm} applies in $\widetilde{Y}$, or $\widetilde{Y}$ has degenerate norm and the claim follows automatically.)  As an inductive hypothesis, assume the theorem holds for $(n-1)$-component links, e.g. $L\setminus L_n\subset Y_{q}(L_n)$. 
Then there is an $(n-3)$-dimensional collection of rays $F$ of $H_2(X_q,\boundary X_q;\R)$ so that if $[S_q]$ is not in a ray in $F$, then $\widehat{S_q}$ is norm-minimizing (where $S_q$ is some properly norm-minimizing surface in $X_q$). For $q\in\Q\cup\{\pm \infty\}\setminus\{0\}$, let $E_q$ contain rays in $H_2(X_q,\boundary X_q;\R)$ whose elements $\beta$ meet $P_n$ in slope $q$ and map to elements of $F$ after capping off $\beta\cap P_n$ in $X_q$. Then $E_q$ is an $(n-3)$-dimensional collection of rays, and the union of all $E_q$ across $q\not\in\{0,\lk_{ij}/(\lk_{in}\lk_{jn})\text{ for some $i\neq j<n$}\}$ is an $(n-2)$-dimensional collection of rays.

Let $E\supset E'$. Then if $[S]$ is not in $E$, we find $\widehat{S}$ is norm-minimizing.

\end{proof}

\bibliographystyle{abbrv}

\end{document}